\documentclass[letterpaper]{article}

\usepackage{amsmath,amsfonts,amsthm,amssymb,mathtools,mathrsfs}
\usepackage{paralist}
\usepackage{array}

\usepackage[algo2e,ruled]{algorithm2e}
\usepackage{algorithmicx, algpseudocode}
\usepackage[margin=1.5in]{geometry}

\usepackage{tikz}
\usepackage{pgf,pgfplots}

\usepackage{graphicx,subcaption}

\newtheorem{theorem}{Theorem}
\newtheorem{corollary}[theorem]{Corollary}
\newtheorem{lemma}[theorem]{Lemma}

\newtheorem{example}[theorem]{Example}

\theoremstyle{definition}
\newtheorem{definition}[theorem]{Definition}
\theoremstyle{remark}

\DeclareMathOperator{\conv}{conv}

\DeclareMathOperator*{\OPT}{OPT}

\title{Interdicting Structured Combinatorial\\
Optimization Problems with $\{0,1\}$-Objectives}

\author{Stephen R.\ Chestnut\thanks{Department of Mathematics, ETH Z\"urich. E-mail:\texttt{stephenc@ethz.ch}.}
\and Rico Zenklusen\thanks{Department of Mathematics, ETH Z\"urich, and
Department of Applied Mathematics and Statistics,
Johns Hopkins University. E-mail:\texttt{ricoz@math.ethz.ch}.} }

\begin{document}

\maketitle

\begin{abstract}
Interdiction problems ask about the worst-case impact
of a limited change to an underlying optimization problem.
They are a natural way to measure the robustness of
a system, or to identify its weakest spots.
Interdiction problems have been studied for a wide
variety of classical combinatorial
optimization problems, including
maximum $s$-$t$ flows, shortest $s$-$t$ paths,
maximum weight matchings, minimum spanning trees,
maximum stable sets, and graph connectivity.
Most interdiction problems are NP-hard, and furthermore,
even designing efficient approximation algorithms that
allow for estimating the order of magnitude of a
worst-case impact, has turned out to be very difficult.
Not very surprisingly, the few known approximation
algorithms are heavily tailored for specific problems.

Inspired by an approach of Burch
et al.~\cite{burch2003decomposition}, we suggest
a general method to obtain pseudoapproximations
for many interdiction problems.
More precisely, for any $\alpha>0$, our algorithm will
return either a $(1+\alpha)$-approximation, or a solution
that may overrun the interdiction budget by a factor
of at most $1+\alpha^{-1}$ but is also at least
as good as the optimal solution that respects the budget.
Furthermore, our approach can handle submodular interdiction
costs when the underlying problem is to find a maximum
weight independent set in a matroid, as for example the
maximum weight forest problem.
Additionally, our approach can sometimes be refined by
exploiting additional structural properties of the underlying 
optimization problem to obtain stronger results.
We demonstrate this by presenting a PTAS for interdicting
$b$-stable sets in bipartite graphs.

\end{abstract}

\section{Introduction}
\label{sec:intro}

One way to understand the robustness of a system is to evaluate
attack strategies.  
This naturally leads to \emph{interdiction} problems; broadly,
one is given a set of feasible solutions, along with some rules
and a budget for modifying the set, with the goal of inhibiting
the solution to an underlying nominal optimization problem. 
A prominent example that nicely
highlights the nature of interdiction problems is maximum
flow interdiction. Here, the nominal problem is a maximum
$s$-$t$ flow problem. Given is a directed graph
$G=(V,A)$ with arc capacities $u:A\rightarrow \mathbb{Z}_{> 0}$,
a source $s\in V$ and sink $t\in V\setminus \{s\}$.
Furthermore, each arc has an \emph{interdiction cost}
$c:A \rightarrow \mathbb{Z}_{>0}$, and there is a global
\emph{interdiction budget} $B\in \mathbb{Z}_{>0}$.
The goal is to find
a subset of arcs $R\subseteq A$ whose
cost does not exceed the interdiction budget, i.e.,
$c(R):=\sum_{a\in R} c(a) \leq B$, such that the
value of a maximum $s$-$t$ flow in the graph $(V,A\setminus R)$
obtained from $G$ by removing $R$ is as small as possible.
In particular, if the value of a maximum $s$-$t$ flow
in $G=(V,E)$ is denoted by $\nu((V,E))$, then we can 
formalize the problem as follows
\begin{equation*}
\min_{R\subseteq A: c(R)\leq B} \nu((V,E\setminus R)).
\end{equation*}
A set $R\subseteq A$ with $c(R)\leq B$ is often called
an \emph{interdiction or removal set}.
Similarly, one can define interdiction problems for
almost any underlying nominal optimization problem.

Interdiction is of practical interest for evaluating robustness
and developing attack strategies. Indeed, even the discovery
of the famous Max-Flow/Min-Cut Theorem was motivated by a Cold
War plan to interdict the Soviet rail network in Eastern
Europe~\cite{schrijver2002history}.
Interdiction has also been studied to find cost-effective
strategies to prevent the spread of infection in a
hospital~\cite{assimakopoulos1987network}, to determine how
to inhibit the distribution of illegal drugs~\cite{wood1993deterministic},
to prevent nuclear arms smuggling~\cite{pan2003stochastic},
and for infrastructure
protection~\cite{salmeron2009worst,church2004identifying},
just to name a few applications.

A significant effort has been dedicated to understanding
interdiction problems.
The list of optimization problems for which interdiction
variants have been studied includes
maximum flow~\cite{wollmer1964removing,wood1993deterministic,%
phillips1993network,zenklusen2010network},
minimum spanning
tree~\cite{frederickson1996increasing,zenklusen2015approximation},
shortest path~\cite{ball1989finding,khachiyan2008short},
connectivity of a graph~\cite{zenklusen_2014_connectivity},
matching~\cite{zenklusen2010matching,pan2013interdiction},
matroid rank~\cite{joret2012reducing,juttner2006budgeted},
stable set~\cite{bazgan_2011_most},
several variants of facility
location~\cite{church2004identifying,bazgan2013complexity}, and more.

Although one can generate new interdiction problems mechanically
from existing optimization problems, there are few general
techniques for their solution.
The lack of strong exact algorithms for interdiction problems in
not surprising in light of the fact that almost all known
interdiction problems are NP-hard.
However, it is intriguing how little is known about
the approximability of interdiction problems.
In the context of interdiction problems, the design
of approximation algorithms is of particular interest
since it often allows accurate estimation of at least
the order of magnitude of a potential worst-case impact,
which turns out to be a nontrivial task in this context.
Polynomial-time approximation schemes (PTASs) are primarily
known only when assuming particular graph structures or
other special cases. In particular, for planar graphs
PTASs have been found for network flow
interdiction~\cite{phillips1993network,zenklusen2010network}
and matching interdiction~\cite{pan2013interdiction}.
Furthermore, PTASs based on pseudopolynomial algorithms
have been obtained for some interdiction problems on
graphs with bounded
treewidth~\cite{zenklusen2010matching,bazgan_2011_most}.
Connectivity interdiction is a rare exception where a
PTAS is known without any further restrictions on the
graph structure~\cite{zenklusen_2014_connectivity}.
Furthermore, $O(1)$-approximations are known for
minimum spanning
tree interdiction~\cite{zenklusen2015approximation},
and for interdicting a class of packing interdiction
problems which implies an $O(1)$-approximation
for matching interdiction~\cite{dinitz2013packing}.
However, for most classical polynomial-time solvable
combinatorial optimization problems, like shortest paths,
maximum flows and maximum matchings, there
is a considerable gap between the approximation quality of
the best known interdiction algorithm and the currently
strongest hardness result. In particular, among the above-mentioned
problems, only the interdiction of shortest $s$-$t$ paths is
known to be APX-hard, and matching interdiction is the only
one among these problems for which
an $O(1)$-approximation is known.
For network
flow interdiction, no approximation results are known, even though
only strong NP-hardness is known from a complexity point of view.

Burch et al.~\cite{burch2003decomposition} decided to go
for a different approach to attack the network flow interdiction
problem, leading to the currently best known solution guarantee
obtainable in polynomial time.
Their algorithm solves a linear programming (LP)
relaxation to find a fractional
interdiction set that lies on an edge of an integral polytope.  
It is guaranteed that, for any $\alpha >0$, one of the vertices
on that edge is either a budget feasible $(1+\alpha)$-approximate
solution or a super-optimal solution that overruns the budget by
at most a factor of $1+1/\alpha$.
However, one cannot predetermine
which objective is approximated and the choice
of $\alpha$ biases the outcome.
For simplicity we call such an algorithm a $2$-pseudoapproximation
since, in particular, by choosing $\alpha=1$ one either gets a
$2$-approximation or a super-optimal solution using at most
twice the budget.
In this context, it is also common to use the notion
of a $(\sigma, \tau)$-approximate solution, for $\sigma, \tau \geq 1$.
This is a solution that violates the budged constraint
by a factor of at most $\tau$, and has a value that is
at most a factor of $\sigma$ larger than the value of an optimal
solution, which is not allowed to violate the budget.
Hence, a $2$-pseudoapproximation is an algorithm that,
for any $\alpha >0$, either returns a
$(1+\alpha, 1)$-approximate solution or a
$(1,1+1/\alpha)$-approximate solution.

\medskip

The main result of this paper is a general technique
to get $2$-pseudoapproximations for a wide set of interdiction
problems.
To apply our technique we need three conditions on the nomial
problem we want to interdict. First, we need to have an 
LP description of the nomial problem that has
a well-structured dual.
In particular, box-total dual integrality (box-TDI)
is sufficient. The precise conditions are described
in Section~\ref{sec:settingRes}.
Second, the LP description of the nomial problem is a maximization
problem whose objective vector only has $\{0,1\}$-coefficients.
Third, the LP description of the nomial problem fulfills
a down-closedness property, which we call \emph{$w$-down-closedness}.
This third condition is fulfilled by all independence systems,
i.e., problems where a subset of a feasible solution is also
feasible, like forests, and further problems like maximum
$s$-$t$ flows. Again, a precise description is given
in Section~\ref{sec:settingRes}.
In particular, our framework leads to $2$-pseudoapproximations
for the interdiction of any problem that asks
to find a maximum cardinality set in an independence system
for which a box-TDI description exists.
This includes maximum cardinality independent set in a
matroid, maximum cardinality common independent set in
two matroids, $b$-stable sets in bipartite graphs,
and more. Furthermore, our conditions also include
the maximum $s$-$t$ flow problem, thus implying the
result of Burch et al.~\cite{burch2003decomposition},
even though $s$-$t$ flows do not form an
independence system.
Apart from its generality, our approach has further
advantages.
When interdicting independent sets of a matroid, we can
even handle general nonnegative objective functions,
instead of only $\{0,1\}$-objectives. This is obtained
by a reformulation of the weighted problem to a 
$\{0,1\}$-objective problem over a polymatroid.
Also, we can get a $2$-pseudoapproximation for
interdicting maximum weight independent sets in a
matroid with \emph{submodular} interdiction costs.
Submodular interdiction costs allow for modeling
economies of scale when interdicting. More precisely,
the cost of interdicting an additional element is the
smaller the more elements will be interdicted. 
Additionally, our approach can sometimes be refined by
exploiting additional structural properties of the underlying 
optimization problem to obtain stronger results.
We demonstrate this by presenting a PTAS for interdicting
$b$-stable sets in bipartite graphs, which is
an NP-hard problem.
We complete the discussion of $b$-stable set interdiction
in bipartite graphs by showing that interdicting classical
stable sets in bipartite graphs, which
are $1$-stable sets, can be done efficiently by a
reduction to matroid intersection.
This generalizes a result by Bazgan, Toubaline and
Tuza~\cite{bazgan_2011_most} who showed that interdiction
of stable sets in a bipartite graph is polynomial-time solvable 
if all interdiction costs are one.

\subsubsection*{Organization of the paper}

In Section~\ref{sec:settingRes},
we formally describe the class of interdiction problems
we consider, together with the technical assumptions
required by our approach, to obtain a
$2$-pseudoapproximation. Furthermore,
Section~\ref{sec:settingRes} also contains a 
formal description of our results.
Our general approach to obtain $2$-pseudoapproximations
for a large set of interdiction problems is described
in Section~\ref{sec:2pseudoapprox}.
In Section~\ref{sec:matroidInt} we show how, in the
context of interdicting independent sets in a matroid,
our approach allows for getting
a $2$-approximation for general nonnegative weights
and submodular interdiction costs.
Section~\ref{sec:bIndep} shows how our approach can
be refined for the interdiction of $b$-stable set
interdiction in bipartite graphs to obtain a PTAS.
Furthermore, we also present an efficient algorithm
for stable set interdiction in bipartite graphs
in Section~\ref{sec:bIndep}.

\section{Problem setting and results}
\label{sec:settingRes}

We assume that feasible solutions to the nominal
problem, like matchings or $s$-$t$ flows,
can be described as follows. There is a finite
set $N$, and the feasible solutions can be described
by a bounded and nonempty
set $\mathcal{X}\subseteq \mathbb{R}^N_{\geq 0}$
such that $\conv(\mathcal{X})$
is an integral polytope\footnote{The discussion that follows
also works for feasible sets $\mathcal{X}$ such that 
$\conv(\mathcal{X})$ is not integral. However, integrality of
$\mathcal{X}$ simplifies parts of our discussion and
is used to show that our $2$-pseudoapproximation is efficient.
Furthermore, all problems we consider naturally have the
property that $\conv(\mathcal{X})$ is integral.
}%
.
For example, for matchings we can choose $N$ to
be the edges of the given graph $G=(V,E)$, and
$\mathcal{X} \subseteq \{0,1\}^E$ are all
characteristic vectors of matchings $M\subseteq E$
in~$G$.
Similarly, consider the maximum $s$-$t$ flow
problem on a directed graph $G=(V,A)$, with
edge capacities $u:A\rightarrow \mathbb{Z}_{> 0}$.
Here, we can choose $N=A$ and
$\mathcal{X}\subseteq \mathbb{R}_{\geq 0}^N$ contains
all vectors $f\in \mathbb{R}_{\geq 0}^N$ that
correspond to $s$-$t$ flows.

Furthermore, the nominal problem should be solvable by maximizing a linear function $w$
over $\mathcal{X}$.
For the case of maximum cardinality matchings one can maximize
the linear function with all coefficients being equal to $1$.
Finally, we assume that we interdict elements of the ground
set $N$, and the interdiction problem can be described by
the following min-max mathematical optimization problem:
\begin{equation}
\begin{array}{>{\displaystyle}rr@{\;\;}c@{\;\;}ll}
\min_{\substack{R\subseteq N:\\ c(R)\leq B}} \max &w^T x &    &  &\\[-1.2em]
                         &    x &\in & \mathcal{X} & \\
                         &  x(e)&  = & 0 & \forall e\in R,
\end{array}
\label{eq:initProb}
\end{equation}
where $c:N\rightarrow \mathbb{Z}_{> 0}$
are interdiction costs on $N$, and $B\in \mathbb{Z}_{> 0}$
is the interdiction budget.
It is instructive to consider matching interdiction
where one can choose $N$ to be all edges and
$\mathcal{X}\subseteq \{0,1\}^N$ the characteristic
vectors of matchings. Imposing $x(e)=0$ then enforces
that one has to choose a matching that does not contain
the edge $e$ which, as desired, corresponds to interdicting $e$.
 
Notice that the above way of describing interdiction problems
is very general. In particular, it contains a large set
of classical combinatorial interdiction problems,
like interdicting maximum $s$-$t$ flows, maximum matchings,
maximum cardinality stable sets of a graph,
maximum weight forest, and more generally, maximum
weight independent set in a matroid or the intersection
of two matroids.

Our framework for designing $2$-pseudoapproximations
for interdiction problems of type~\eqref{eq:initProb}
requires the following three properties, on which we
will expand in the following:

\begin{enumerate}[(i)]
\item\label{item:01vec}
The objective vector $w$ is a $\{0,1\}$-vector,
i.e., $w\in \{0,1\}^N$,

\item\label{item:wDownClosed}
the feasible set $\mathcal{X}$ is
\emph{$w$-down-closed}, which is a weaker
form of down-closedness that we introduce below, and

\item\label{item:wDIsolvable}
there is a linear description of the
convex hull $\conv(\mathcal{X})$ of $\mathcal{X}$
which is \emph{box-$w$-DI solvable}.
This is a weaker form of being box-TDI equipped
with an oracle that returns an integral dual
solution to box-constrained linear programs over the
description of $\conv(\mathcal{X})$.
\end{enumerate}

In the following we formally define the second
and third condition, by giving precise definitions
of $w$-down-closedness and box-$w$-DI solvability.
In particular, condition~\eqref{item:wDIsolvable},
i.e., box-$w$-DI solvability, describes how we can
access the nominal problem.

\subsection{$w$-down-closedness}

The notion of $w$-down-closedness is a weaker
form of down-closedness. We recall that a
set $\mathcal{X}\subseteq \mathbb{R}^N_{\geq 0}$
is down-closed if for any $x\in \mathcal{X}$
and $y\in \mathbb{R}^N_{\geq 0}$ with 
$y\leq x$ (componentwise), we have $y\in \mathcal{X}$.
Contrary to the usual notion of down-closedness,
$w$-down-closedness depends on the $\{0,1\}$-objective
vector $w$.

\begin{definition}[$w$-down-closedness]
\label{def:wDownClosed}
Let $w\in \{0,1\}^N$. 
$\mathcal{X}\subseteq \mathbb{R}^N_{\geq 0}$ is
\emph{$w$-down-closed} if for every $x\in \mathcal{X}$ and
$e\in N$ with $x(e)>0$, there exists $x' \leq x$ such that
the following conditions hold:
\begin{enumerate}[(i)]
\item $x'\in \mathcal{X}$;
\item $x'(e) = 0$;
\item\label{defitem:wDownRed} $w^T x' \geq w^T x - x(e)$.
\end{enumerate}
\end{definition}

Notice that if $\mathcal{X}\subseteq \mathbb{R}_{\geq 0}^N$
is down-closed, then it is
$w$-down-closed for any $w\in \{0,1\}^N$, since
one can define $x'\in \mathcal{X}$ in the above
definition by
$x'(f)=x(f)$ for $f\in N\setminus \{e\}$
and $x'(e)=0$.
Similarly, $w$-down-closedness also includes all
independence systems. We recall that an
\emph{independence system} 
over a ground set $N$ is a family $\mathcal{F}\subseteq 2^N$
of subsets of $N$ such that for any $I\in \mathcal{F}$
and $J\subseteq I$, we have $J\in \mathcal{F}$. In
other words, it is closed under taking subsets.
Typical examples of independence systems include
matchings, forests and stable sets.
Naturally, an independence
system $\mathcal{F}\subseteq 2^N$ can be represented
in $\mathbb{R}^N_{\geq 0}$ by its characteristic vectors,
i.e., $\mathcal{X}=\{\chi^I \mid I\in \mathcal{F}\}$,
where $\chi^I\in \{0,1\}^N$ denotes the characteristic
vector of $I$.
Clearly, for the same reasons as for down-closed sets,
the set $\mathcal{X}$ of characteristic vectors of any
independence system is $w$-down-closed for any
$w\in \{0,1\}^N$.

Hence, many natural combinatorial optimization problems
are $w$-down-closed for any $w\in \{0,1\}^N$,
including matchings, stable sets,
independent sets in a matroid or the intersection
of two matroids. 
Furthermore, $w$-down-closedness also captures the
maximum $s$-$t$ flow problem, and a generalization
of it,
known as \emph{polymatroidal network flows},
that was introduced independently by
Hassin~\cite{hassin_1978_network}
and Lawler and Martel~\cite{lawler_1982_computing}.
Loosely speaking, polymatroidal network flows correspond
to classic flows with, for every vertex, the addition of
submodular packing constraints on the incoming arcs as well as the outgoing ones.
See~\cite{lawler_1982_computing} for a formal definition.

\begin{example}[$w$-down-closedness of $s$-$t$
flow polytope]
Let $G=(V,A)$ be a directed graph with two
distinct vertices $s,t\in V$ and arc
capacities $u:A \rightarrow \mathbb{Z}_{>0}$.
Furthermore, we assume that there are
no arcs entering the source $s$, since
such arcs can be deleted when seeking
maximum $s$-$t$ flows.
The $s$-$t$ flow polytope
$\mathcal{X}\subseteq \mathbb{R}^A_{\geq 0}$
can then be described as follows (see,
e.g.,~\cite{korte_2012_combinatorial}):
\begin{equation*}
\mathcal{X} = \left\{x\in \mathbb{R}^A_{\geq 0} \mid
x(\delta^+(v))-x(\delta^-(v)) = 0 \;\forall v\in V\setminus \{s,t\}
\right\},
\end{equation*}
where $\delta^+(v),\delta^-(v)$ denote the
set of arcs going out of $v$ and entering
$v$, respectively; furthermore,
$x(U):=\sum_{a\in U} x(a)$ for $U\subseteq A$.
A maximum $s$-$t$ flow can be found by
maximizing the linear function
$w^T x$ over $\mathcal{X}$, where $w=\chi^{\delta^+(s)}$,
i.e., $w\in \{0,1\}^A$ has a $1$-entry for each
arc $a\in \delta^+(s)$, and $0$-entries for all
other arcs. This maximizes the total outflow
of $s$.
Notice that the value of a flow $x\in \mathcal{X}$ is
equal to $x(\delta^+(s))-x(\delta^-(s))=x(\delta^+(s))$,
since there are no arcs entering $s$; this is indeed
the total outflow of $s$.

To see that $\mathcal{X}$ is $w$-down-closed,
let $x\in \mathcal{X}$ and $e\in A$, and we construct
$x'\in \mathcal{X}$ satisfying the conditions
of Definition~\ref{def:wDownClosed} as follows.
We compute a path-decomposition of $x$
with few terms.
This is a family of $s$-$t$ paths
$P_1,\dots, P_k \subseteq A$ with $k \leq |A|$
together with positive coefficients
$\lambda_1,\dots, \lambda_k >0$ such that
$x=\sum_{i=1}^k \lambda_i \chi^{P_i}$
(see~\cite{ahuja_1993_network} for more details).
Let $I=\{i\in [k] \mid e \in P_i\}$, where
$[k]:=\{1,\dots, k\}$,
and we set $x'=\sum_{i\in [k]\setminus I}\lambda_i \chi^{P_i}$.
The flow $x'\in \mathcal{X}$ indeed satisfies
the conditions of Definition~\ref{def:wDownClosed}.
This follows from the fact that 
$x(e)=\sum_{i\in I} \lambda_i$, and
each path $P_i$ contains precisely
one arc of $\delta^+(s)$, hence,
$x'(\delta^+(s))=x(\delta^+(s)) - \sum_{i\in I} \lambda_i$.
\end{example}

The polytope that corresponds to polymatroidal network
flows (see~\cite{lawler_1982_computing}),
is an $s$-$t$ flow polytope
with additional packing constraints. Its $w$-down-closedness
follows therefore from the $w$-down-closedness of the
$s$-$t$ flow polytope.

\medskip

Furthermore, notice that a non-empty $w$-down-closed system $\mathcal{X}$
always contains the zero vector, independent of $w\in \{0,1\}^N$.
By $w$-down-closedness we can go through all elements $e\in N$
one-by-one, and replace $x$ by a vector $x'\in \mathcal{X}$
with $x'(e)=0$, thus proving that the zero vector is in $\mathcal{X}$.

\subsection{box-$w$-DI solvability}

To obtain $2$-pseudoapproximations for interdiction
problems of type~\eqref{eq:initProb}, we additionally need
to have a good description of the convex hull $\conv(\mathcal{X})$
of $\mathcal{X}$. The type of description we need is 
a weaker form of box-TDI-ness together with an efficient
optimization oracle for the dual that returns integral solutions,
which we call box-$w$-DI solvability, where ``DI'' stands for
`dual integral''.

\begin{definition}[box-$w$-DI solvability]
A description
$\{x\in \mathbb{R}^N \mid Ax\leq b, x\geq 0\}$
of a nonempty polytope $P$ is \emph{box-$w$-DI solvable}
for some vector $w\in \{0,1\}^N$ if the following
conditions hold:
\begin{enumerate}[(i)]
\item\label{item:dualInt}
For any vector $u\in \mathbb{R}^N_{\geq 0}$,
the following linear program has an integral dual
solution if it is feasible:
\begin{equation}
\begin{array}{>{\displaystyle}rr@{\;\;}c@{\;\;}l}
\max &w^T x    &     & \\
     &A x      &\leq &b \\
     &x        &\leq &u \\
     &x        &\geq &0 \\
\end{array}
\label{eq:box-LP}
\end{equation}
Notice that the dual of the above LP is the following
LP:
\begin{equation}
\begin{array}{>{\displaystyle}rr@{\;}r@{\;}r@{\;\;}c@{\;\;}l}
\min          &b^T y  &+& u^T r &     &  \\
              &A^T y  &+&     r &\geq &w \\
              &y      & &       &\geq &0 \\
              &       & &     r &\geq &0 \\
\end{array}
\label{eq:box-DP}
\end{equation}

\item\label{item:dualOracle}
For any $u\in \mathbb{R}^N_{\geq 0}$,
one can decide in polynomial time whether~\eqref{eq:box-LP}
is feasible. Furthermore, if~\eqref{eq:box-LP} is feasible,
one can efficiently compute its objective value
and an integral vector
$r\in \mathbb{Z}^N_{\geq 0}$ that corresponds
to an optimal integral solution to~\eqref{eq:box-DP}, i.e.,
there exists an integral vector $y$ such that
$y,r$ is an integral optimal solution to~\eqref{eq:box-DP}.
\end{enumerate}

\end{definition}

We emphasize that box-$w$-DI solvability does
not assume that the full system
$Ax \leq b , x\geq 0$ is given as input. In particular, this is
useful when dealing with combinatorial problems
whose feasible set
$\mathcal{X}\subseteq \mathbb{R}^N_{\geq 0}$ is such that
the polytope $\conv(\mathcal{X})$ has an exponential number
of facets, and a description of $\conv(\mathcal{X})$
therefore needs an exponential number of
constraints\footnote{In some cases one can get around this
problem by using an \emph{extended formulation}. This is
a lifting of a polytope in a higher dimension with the
goal to obtain a lifted polytope with an inequality
description of only polynomial size
(see~\cite{kaibel_2011_extended,conforti_2013_extended}).}.
Since the only access to $\mathcal{X}$ that we need is
an oracle returning an optimal integral dual solution
to~\eqref{eq:box-DP}, we can typically deal
with such cases if we have an implicit description of
the system $Ax \leq b, x\geq 0$ over which we can
separate with a separation oracle.

Furthermore, notice that condition~\eqref{item:dualInt}
of box-$w$-DI solvability is a weaker form of
box-TDIness due to two reasons. First, our objective vector
$w\in\{0,1\}^N$ is fixed, whereas in box-TDIness, dual
integrality has to hold for any integral objective vector.
Second, when dealing with box-TDIness, one can additionally
add lower bounds $x\geq \ell$ on $x$ in \eqref{eq:box-LP}, still
getting a linear program with an optimal integral
dual solution.

We even have box-TDI descriptions for all problems we discuss here.
The only additional property needed for a box-TDI system
to be box-$w$-DI solvable, is that one can efficiently
find an optimal integral dual solution. However, such
procedures are known for essentially all classical
box-TDI systems.
In particular, this applies to the classical polyhedral
descriptions of the independent sets of a matroid or the
intersection of two matroids, stable sets in bipartite
graphs, $s$-$t$ flows, and any problem whose constraint
matrix can be chosen to be totally unimodular (TU)
and of polynomial size.

Since our only access to the feasible set is via
the oracle guaranteed by box-$w$-DI solvability,
we have to be clear about what we consider to be
the input size when talking about polynomial time
algorithms.
In addition to the binary encodings
of $B$, $c$, we also
assume that the binary encodings of the
optimal value of~\eqref{eq:box-DP} and
the integral optimal vector $r\in\mathbb{Z}^N_{\geq 0}$
returned by the box-$w$-DI oracle are part of
the input size.
This implies that in particular, the
binary encoding 
of $\nu^* = \max\{w^T x \mid Ax \leq b, x\geq 0\}$
is part of the input size.

\subsection{Our results}

The following theorem summarizes our main result for
obtaining $2$-pseudoapproximations.

\begin{theorem}\label{thm:2pseudoapprox}
There is an efficient $2$-pseudoapproximation for
any interdiction problem of type~\eqref{eq:initProb}
if the following conditions are satisfied:
\begin{enumerate}[(i)]
\item The objective function $w$ is a $\{0,1\}$-vector,
i.e., $w\in \{0,1\}^N$,

\item the description of the feasible set
$\mathcal{X}\subseteq \mathbb{R}^N$ is $w$-down-closed, and

\item there is a box-$w$-DI solvable description of
$\conv(\mathcal{X})$.
\end{enumerate}
\end{theorem}

Using well-known box-TDI description of classical
combinatorial optimization
problems~(see~\cite{schrijver2003combinatorial}),
Theorem~\ref{thm:2pseudoapprox} leads to
$2$-pseudoapproximations for the interdiction of
many combinatorial optimization problems.

\begin{corollary}
There is a $2$-pseudoapproximation for interdicting
maximum cardinality independent sets of a matroid
or the intersection of two matroids, maximum
$s$-$t$ flows, and maximum polymatroidal network flows.
Furthermore, there is
a $2$-pseudoapproximation for all problems where
a maximum cardinality set has to be found with
respect to down-closed constraints captured by a TU matrix.
For example, this includes maximum $b$-stable
sets in bipartite graphs. 
\end{corollary}
We recall that for the maximum $s$-$t$ flow problem,
a $2$-pseudoapproximation was already known due to
Burch et al.~\cite{burch2003decomposition}.

\bigskip

Furthermore, for interdicting independent sets of
a matroid we obtain stronger results by
leveraging the strong combinatorial
structure of matroids to adapt our approach.
Consider a matroid $M=(N,\mathcal{I})$ on ground
set $N$ with independent sets
$\mathcal{I}\subseteq 2^N$. We recall the definition
of a matroid, which requires $\mathcal{I}$ to be a
nonempty set such that: (i) $\mathcal{I}$ is
an independence system, i.e., $I\in \mathcal{I}$
and $J\subseteq I$ implies $J\in \mathcal{I}$,
and
(ii) for any $I,J\in \mathcal{I}$ with $|I|<|J|$,
there exists $e\in J\setminus I$ such that
$I\cup\{e\}\in \mathcal{I}$.
We typically assume that a matroid is given by an
\emph{independence oracle}, which is an oracle that,
for any $I\subseteq N$, returns whether $I\in \mathcal{I}$
or not.
See~\cite[Volume~B]{schrijver2003combinatorial} for more
information on matroids.

For matroids, we can get a $2$-pseudoapproximation
even for arbitrary nonnegative weight functions $w$, i.e., 
for interdicting the \emph{maximum weight} independent
set of a matroid. Furthermore, we can also handle
monotone nonnegative submodular interdiction costs $c$. A
\emph{submodular function} $c$ defined on a ground set
$N$, is a function $c:2^N\rightarrow \mathbb{R}_{\geq 0}$
that assigns a nonnegative value $c(S)$ to each set
$S\subseteq N$ and fulfills the following property
of \emph{economies of scale}:
\begin{equation*}
c(A\cup\{e\}) - c(A) \geq c(B\cup \{e\})- c(B)
\qquad A\subseteq B \subseteq N, e\in N\setminus B.
\end{equation*}
In words, the marginal cost of interdicting an
element is lower when more elements will be
interdicted.
Economies of scale can often be a natural property in
interdiction problems. It allows for
modeling dependencies that are sometimes called
\emph{cascading failures} or \emph{chain-reactions},
depending on the context. More precisely, it may
be that the interdiction of a set of elements $S\subseteq N$
will render another element $e\in N$ unusable.
This can be described by a submodular interdiction
cost $c$ which assigns a marginal cost of $0$ to the
element $e$, once all elements of $S$ have been removed.
Still, removing only $e$ may have a strictly positive
interdiction cost.
Such effects cannot be captured with linear
interdiction costs.
A submodular function $c:2^N\rightarrow \mathbb{R}_{\geq 0}$
is called \emph{monotone} if $c(A)\leq c(B)$ for
$A\subseteq B\subseteq N$.
We typically assume that a submodular function $f$ is
given through a \emph{value oracle}, which is an 
oracle that, for any set $S\subseteq N$, returns
$f(S)$.

\begin{theorem}
There is an efficient $2$-pseudoapproximation
to interdict the problem of finding a maximum weight
independent set in a matroid, with monotone
nonnegative submodular interdiction
costs. The following is a formal description of this
interdiction problem:
\begin{equation*}
\min_R\{ \max_I \{ w(I) \mid I\in\mathcal{I},
I\cap R =\emptyset\} \mid R\subseteq N, c(R)\leq B\},
\end{equation*}
where $c:2^N\rightarrow \mathbb{R}_{\geq 0}$ is
a monotone nonegative submodular function, and
$w\in \mathbb{Z}_{\geq 0}^N$ \footnote{Notice that
the integrality requirement for $w$ is not restrictive.
Any $w\in \mathbb{Q}_{\geq 0}^N$ can be scaled up to
an integral weight vector without changing the problem.}.
The matroid is given through an independence oracle and
the submodular cost function $c$ through a value oracle.
\end{theorem}

\medskip

Finally, we show that our approach can sometimes be
refined to obtain stronger approximation guarantees.
We illustrate this on the interdiction version
of the $b$-stable set problem in bipartite graphs.
Here, a bipartite graph $G=(V, E)$  with bipartition
$V=I\cup J$, and a vector $b\in \mathbb{Z}^E_{> 0}$
is given. A $b$-stable set in $G$
is a vector $x\in \mathbb{Z}_{\geq 0}^V$ such that
$x(i) + x(j) \leq b(\{i,j\})$ for $\{i,j\}\in E$.
Hence, by choosing $b$ to be the all-ones vector,
we obtain the classical stable set problem.
Because it can be formulated as a linear program with TU constraints,
finding a maximum cardinality $b$-stable set
in a bipartite graph is efficiently solvable.
However, its interdiction version is easily seen to be
NP-hard by a reduction from the knapsack problem.
Exploiting the adjacency properties of
a polytope that is crucial in our analysis
we can even get a true approximation algorithm,
which does not violate the budget.
More precisely, we obtain a polynomial-time
approximation scheme (PTAS), which is an algorithm
that, for any $\epsilon > 0$, computes efficiently an
interdiction set leading to a value of at
most $1-\epsilon$ times the optimal value.
\begin{theorem}\label{thm:PTASBIndepSet}
There is a PTAS for the interdiction of
$b$-stable sets in bipartite graphs.
\end{theorem}

We complete this discussion of interdicting
$b$-stable sets in bipartite graphs
by showing that the special case of
interdicting stable sets in bipartite
graph, i.e., $b=1$, is efficiently solvable.
This is done through a reduction to a
polynomial number of efficiently solvable
matroid intersection problems.

\begin{theorem}\label{thm:maxCardBIndepSet}
The problem of interdicting the maximum cardinality
stable set in a bipartite graph can be
solved efficiently.
\end{theorem}

The above theorem generalizes a result by
Bazgan, Toubaline and Tuza~\cite{bazgan_2011_most}
who showed that interdiction of stable sets
in bipartite graphs can be done efficiently when
all interdiction costs are one.
Our result applies to arbitrary interdiction costs.

\section{General approach to obtain $2$-pseudoapproximations}
\label{sec:2pseudoapprox}

Consider an interdiction problem that fulfills the conditions
of Theorem~\ref{thm:2pseudoapprox}.
As usual, let $N$ be the ground set of our problem,
$w\in \{0,1\}^N$ be the objective vector, and we denote
by $\mathcal{X}\subseteq \mathbb{R}^N_{\geq 0}$ the set of 
feasible solutions. Furthermore, let
$\{x\in \mathbb{R}^N \mid Ax \leq b, x \geq 0\}
= \conv(\mathcal{X})$ be a box-$w$-DI solvable
description of $\conv(\mathcal{X})$. We denote
by $m$ the number of rows of $A$.

One key ingredient in our approach is to model interdiction
as a modification of the objective instead of a restriction
of sets that can be chosen.
This is possible due to $w$-down-closedness.
More precisely, we replace the description of the
interdiction problem given by~\eqref{eq:initProb} with
the following min-max problem.

\begin{equation}
\begin{array}{>{\displaystyle}rr@{\;\;}c@{\;\;}l}
\min_{r} \max_{x} &(w-r)^T x & & \\
              &A x      &\leq &b \\
              &x        &\geq &0 \\
              &   c^T r &\leq &B \\
              &       r &\in  &\{0,1\}^N \\
\end{array}
\label{eq:minMaxExact}
\end{equation}

We start by showing that~\eqref{eq:initProb}
and~\eqref{eq:minMaxExact} are equivalent
in the following sense.
For any interdiction set $R\subseteq N$, let
\begin{align*}
\phi(R):&=
\max\{w^T x \mid x\in \conv(\mathcal{X}), x(e)=0 \;\forall e\in R\} \\
&=\max\{w^T x \mid Ax \leq b,
  x\geq 0, x(e)=0 \;\forall e\in R\}.
\end{align*}
Hence, $\phi(R)$ is the value of the
problem~\eqref{eq:initProb} for a fixed set $R$.
Similarly, we define for any characteristic vector
$r\in\{0,1\}^N$ of an interdiction set
\begin{align*}
\psi(r) :&= \max\{(w-r)^T x \mid x\in \conv(\mathcal{X})\}\\
&= \max\{(w-r)^T x \mid Ax \leq b, x\geq 0\}.
\end{align*}
Thus, $\psi(r)$ is the value of~\eqref{eq:minMaxExact}
for a fixed vector $r\in\{0,1\}^N$.

\begin{lemma}\label{lem:intObj}
For every interdiction set $R\subseteq N$, we have
$\phi(R) = \psi(\chi^R)$.
In particular, this implies that~\eqref{eq:initProb}
and~\eqref{eq:minMaxExact} have the same optimal
value, and optimal interdiction sets $R$ to~\eqref{eq:initProb}
correspond to optimal characteristic vectors $\chi^R$
to~\eqref{eq:minMaxExact} and vice versa.
\end{lemma}

We show Lemma~\ref{lem:intObj} based on another lemma
stated below that highlights an important
consequence of $w$-down-closedness,
which we will use later again.

\begin{lemma}\label{lem:clipR}
Let $r\in \mathbb{R}^N_{\geq 0}$ and
$U = \{e\in N \mid r(e) \geq 1\}$.
Then there exists $x\in \mathbb{R}_{\geq 0}^N$ with
$x(e) = 0\;\forall e\in U$, such that $x$ is an optimal
solution to the following linear program.
\begin{equation}
\begin{array}{>{\displaystyle}rr@{\;\;}c@{\;\;}l}
\max_{x} &(w-r)^T x & & \\
         &A x      &\leq &b \\
         &x        &\geq &0 \\
\end{array}
\label{eq:psiProb}
\end{equation}

\end{lemma}
\begin{proof}
Among all optimal solutions to the above linear
program, let $x^*$ be one that minimizes
$x^*(U)$.
Notice that $x^*$ can be chosen to be a vertex
of $\conv(\mathcal{X}) = \{x\in
\mathbb{R}^N \mid Ax \leq b , x\geq 0\}$,
since $x^*$ can be obtained by minimizing the
objective $\chi^U$ over the face of all optimal
solutions to the above LP.
We have to show $x^*(U) = 0$.
Assume for the sake of contradiction that
there is an element $e\in U$ such that
$x^*(e) > 0$.
Since $x^*$ is a vertex of $\conv(\mathcal{X})$,
we have $x^*\in \mathcal{X}$.
By $w$-down-closedness of $\mathcal{X}$, there
is a vector $x' \in \mathcal{X}$ with $x' \leq x^*$,
$x'(e)=0$, and $w^T x' \geq w^T x^* - x^*(e)$.
We thus obtain
\begin{align*}
(w - r)^T x' &\geq w^T x^* - x^*(e) - r^T x' \\
             &\geq w^T x^* - x^*(e) - (r^T x^* - x^*(e)) \\
             &= (w-r)^T x^*,
\end{align*}
where in the second inequality we used
$r^T x' \leq r^T x^* - x^*(e)$, which follows
from $x'\leq x^*$ together with $x'(e)=0$
and $r(e) \geq 1$.
Hence, $x'$ is an optimal solution to the LP
with $x'(U) < x^*(U)$, which violates the
definition of $x^*$ and thus finishes the proof.
\end{proof}

\begin{proof}[Proof of Lemma~\ref{lem:intObj}]
Let $R\subseteq N$ be an interdiction set, and $r=\chi^R$
its characteristic vector.
Let $x\in \mathcal{X}$ be an optimal solution to the
maximization problem defining $\phi(R)$, i.e.,
$w^T x = \phi(R)$ and $x(e)=0 \;\forall e\in R$.
We have
\begin{equation*}
\psi(r) \geq (w-r)^T x = w^T x = \phi(R),
\end{equation*}
where the first equality follows from $x(e)=0$
for $e\in R$. Hence, $\psi(r) \geq \phi(R)$.

Conversely, let $x\in \conv(\mathcal{X})$
be an optimal
solution to the maximization problem defining
$\psi(r)$, i.e., $\psi(r) = (w-r)^T x$.
By Lemma~\ref{lem:clipR}, $x$ can be chosen
such that $r^T x = x(R) = 0$. Hence,
\begin{equation*}
\psi(r) = (w-r)^T x = w^T x \leq \phi(R),
\end{equation*}
and thus $\phi(R) = \psi(r)$.
\end{proof}

Hence,~\eqref{eq:minMaxExact} is an alternative description
of the interdiction problem~\eqref{eq:initProb} in which
we are interested.
In a next step we relax the integrality of $r$ to obtain
the following mathematical program.

\begin{equation}
\begin{array}{>{\displaystyle}rr@{\;\;}c@{\;\;}l}
\min_{r\in \mathbb{R}^n} \max_{x \in \mathbb{R}^n} &(w-r)^T x & & \\
              &A x      &\leq &b \\
              &x        &\geq &0 \\
              &   c^T r &\leq &B \\
              &1 \;\; \geq \;\; r        &\geq &0 \\
\end{array}
\label{eq:minMax}
\end{equation}
As we will show next, the constraint $1\geq r$ can be
dropped due to $w$-down-closedness without changing
the objective.
This leads to the following problem.
\begin{equation}
\begin{array}{>{\displaystyle}rr@{\;\;}c@{\;\;}l}
\min_{r\in \mathbb{R}^n} \max_{x \in \mathbb{R}^n} &(w-r)^T x & & \\
              &A x      &\leq &b \\
              &x        &\geq &0 \\
              &   c^T r &\leq &B \\
              &r        &\geq &0 \\
\end{array}
\label{eq:minMaxSimple}
\end{equation}
The following lemma not only highlights that the
objective values of~\eqref{eq:minMax}
and~\eqref{eq:minMaxSimple} match, but also
that any optimal interdiction vector $r$
of~\eqref{eq:minMaxSimple} can easily be transformed
to an optimal interdiction vector of~\eqref{eq:minMax}.
Thus, we can restrict
ourselves to~\eqref{eq:minMaxSimple}.
We recall that $\psi(r)$ corresponds to the inner
maximization problem of both~\eqref{eq:minMax}
and~\eqref{eq:minMaxSimple} for a fixed vector $r$.

\begin{lemma}\label{lem:dropUpperBound1}
We have
\begin{equation*}
\psi(r) = \psi(r\wedge 1)
  \qquad \forall r\in \mathbb{R}_{\geq 0}^N,
\end{equation*}
where $r\wedge 1$ is the component-wise minimum between
$r$ and the all-ones vector $1\in \mathbb{R}^N$.
This implies that~\eqref{eq:minMax}
and~\eqref{eq:minMaxSimple} have
the same optimal value, and if $r$ is optimal
for~\eqref{eq:minMaxSimple} then 
$r\wedge 1$ is optimal for~\eqref{eq:minMax}.
\end{lemma}
\begin{proof}
Let $r\in \mathbb{R}^N_{\geq 0}$ and consider the
maximization problem that defines $\psi(r)$, which
is the same as the linear program desribed
by~\eqref{eq:psiProb}.
Furthermore, let $r'=r\wedge 1$, and let
$U=\{e\in N \mid r(e) \geq 1\}$.
In particular, $r$ and $r'$ are identical on $N\setminus U$.
We clearly have $\psi(r') \geq \psi(r)$ by monotonicity
of $\psi$. Therefore only $\psi(r') \leq \psi(r)$ has
to be shown.

By Lemma~\ref{lem:clipR}, there exists an optimal
vector $x$ to the maximization problem defining $\psi(r')$
that satisfies $x(e) = 0 \; \forall e\in U$.
Furthermore, by using that $r$ and $r'$ are
identical on $N\setminus U$ we obtain
\begin{align*}
\psi(r') &= (w-r')^T x \\
         &= w^T x - \sum_{e\in N\setminus U} r'(e) x(e)
                  - \sum_{e\in U} r'(e) x(e)\\
         &= w^T x - \sum_{e\in N\setminus U} r(e) x(e)
                  - \sum_{e\in U} r'(e) x(e)
           && (\text{$r$ and $r'$ are identical on $N\setminus U$})\\
         &= w^T x - \sum_{e\in N\setminus U} r(e) x(e)
                  - \sum_{e\in U} r(e) x(e)
           && (\text{$x(e)=0$ for $e\in U$})\\
         &= (w-r)^T x \\
         &\leq \psi(r),
\end{align*}
as desired.
\end{proof}

Interestingly, problem~\eqref{eq:minMaxSimple} has already
been studied in a different context. It can be interpreted
as the problem to inhibit a linear optimization problem by
a continuous and limited change of the objective vector $w$.
In particular, Frederickson and
Solis-Oba~\cite{frederickson1996increasing,%
frederickson1997efficient} present efficient
algorithms to solve this problem when the underlying
combinatorial problem is the maximum weight independent
set problem in a matroid.
J{\"u}ttner~\cite{juttner2006budgeted} presents
efficient procedures for
polymatroid intersection and minimum cost circulation problem.
Also, J{\"u}ttner provides an excellent discussion
how such problems can be solved efficiently using parametric
search techniques.

However, our final goal is quite different from their setting
since, eventually,
we need to find a $\{0,1\}$-vector $r$.
This difference is underlined by the fact that without
integrality, problem~\eqref{eq:minMaxSimple}
can often be solved efficiently,
whereas the interdiction problems we consider are NP-hard.

Still, we continue to
further simplify~\eqref{eq:minMaxSimple}
in a similar way as it was done by
J{\"u}ttner~\cite{juttner2006budgeted}.
For a fixed $r$, the inner maximization problem
in~\eqref{eq:minMaxSimple} is a
linear program with a finite optimum, since
$\mathcal{X}$ is bounded and nonempty by
assumption, and therefore also
$\conv(\mathcal{X})=\{x\in \mathbb{R}^N\mid Ax\leq b, x\geq 0\}$
is bounded and nonempty.
Hence, we can leverage strong duality 
to dualize the inner maximization into a minimization problem.
We thus end up with a problem where we first
minimize over $r$ and then over the dual variables,
which we can rewrite as a single minimization,
thus obtaining the following LP.
\begin{equation}
\begin{array}{>{\displaystyle}rr@{\;}r@{\;\;}c@{\;\;}l}
\min          &b^T y  &     &     &  \\
              &A^T y  &+\;r &\geq &w \\
              &y   &        &\geq &0 \\
              &    & c^T r  &\leq &B \\
              &    &     r  &\geq &0 \\
\end{array}
\label{eq:dualLPfull}
\end{equation}
Hence, by strong duality, the optimal value
of~\eqref{eq:dualLPfull} is the same as the optimal
value of~\eqref{eq:minMaxSimple}.
This reduction also shows why
problem~\eqref{eq:minMaxSimple}, which has no integrality
constraints on $r$, can often be solved
efficiently. This can often be achieved by
obtaining an optimal vector $r\in \mathbb{R}^N_{\geq 0}$
by solving the LP~\eqref{eq:dualLPfull} with standard
linear programming techniques.

What we will do in the following is to show that there
is an optimal solution $(r,y)$ for~\eqref{eq:minMaxSimple} which
can be written as a convex combination of two integral
solutions $(r^1,y^1)$ and $(r^2,y^2)$ that may violate
the budget constraint.
Similar to a reasoning used in Burch et
al.~\cite{burch2003decomposition}
this then implies than one of $r^1$ and $r^2$ is
a $2$-pseudoapproximation.

To compute $r^1$ and $r^2$, we move the constraint
$c^T r \leq B$ in~\eqref{eq:dualLPfull} into the objective
via Lagrangian duality, by introducing a
multiplier $\lambda\geq 0$
(see~\cite{boyd_2004_convex} for more details).
We do this in two steps to highlight that the resulting
Lagrangian dual problem can be solved via the oracle
guaranteed by box-$w$-DI solvability.
First, we dualize~\eqref{eq:dualLPfull} to the 
obtain the following linear program, which is
nicely structured in the sense
that for any fixed $\lambda\geq 0$, it corresponds to
optimizing a linear function over $\conv(\mathcal{X})$
with upper box constraints.

\begin{equation}
\begin{array}{>{\displaystyle}rr@{\;}r@{\;\;}c@{\;\;}l}
\max        &w^T z  & -\;\lambda B  &     &  \\
            &A z    &             &\leq &b \\
            &z      & -\;\lambda c  &\leq &0 \\
            &z      &         &\geq &0 \\
            &       & \lambda &\geq &0 \\
\end{array}
\label{eq:primLPfull}
\end{equation}

Consider the above LP as a problem parameterized by
$\lambda\geq 0$.
Since $\{x\in \mathbb{R}^N \mid Ax \leq b, x\geq 0\}$
is box-$w$-DI solvable, the LP obtained from~\eqref{eq:primLPfull}
by fixing $\lambda\geq 0$ has an optimal integral dual
solution. Furthermore, such an optimal integral dual solution 
can be found efficiently by box-$w$-DI solvability.
The dual problem of~\eqref{eq:primLPfull} for a fixed
$\lambda \geq 0$ is the problem LP($\lambda$) below
with optimal objective value $L(\lambda)$.
\begin{equation}
\begin{array}{>{\displaystyle}rr@{\;}r@{\;\;}c@{\;\;}l}
L(\lambda) =
        \min  &b^T y  &\multicolumn{3}{@{}l}{-\;\lambda(B-c^T r)}   \\
              &A^T y  &+\;r &\geq &w \\
              &y   &        &\geq &0 \\
              &    &    r   &\geq &0 \\
\end{array}
\tag{LP($\lambda$)}
\label{eq:dualLPlag}
\end{equation}
Notice that~\ref{eq:dualLPlag} is indeed the problem
obtained from~\eqref{eq:dualLPfull} by moving the
constraint $c^T r \leq B$ into the objective using $\lambda$
as Lagrangian multiplier.

The following lemma summarizes the relationships between
the different problems we introduced.

\begin{lemma}\label{lem:optValues}
The optimal values of~\eqref{eq:minMax},
\eqref{eq:minMaxSimple}, \eqref{eq:dualLPfull}, and
\eqref{eq:primLPfull} are all the same and equal
to $\max_{\lambda \geq 0} L(\lambda)$.

Furthermore, the common optimal value of the above-mentioned
problems are a lower bound to $\OPT$, the optimal value of
the considered interdiction problem~\eqref{eq:initProb}.
\end{lemma}
\begin{proof}
Problem~\eqref{eq:minMax} and~\eqref{eq:minMaxSimple} have
identical optimal values due to Lemma~\ref{lem:dropUpperBound1}.
The LP~\eqref{eq:dualLPfull} was obtained
from~\eqref{eq:minMaxSimple} by dualizing the inner
maximization problem. Both problems have the same optimal
value due to strong duality, which holds since
$\conv(\mathcal{X})=\{x\in \mathbb{R}^N \mid Ax\leq b, x\geq 0\}$
is a nonempty polytope and thus, the inner maximization
problem of~\eqref{eq:minMaxSimple} has a finite optimum
value for any $r\in \mathbb{R}^n$.
This also shows that the optimum value
of~\eqref{eq:minMaxSimple}, and hence also of \eqref{eq:dualLPfull}
and \eqref{eq:minMax}, is finite.
Problems~\eqref{eq:primLPfull} and~\eqref{eq:dualLPfull} are
a primal-dual pair of linear programs. For this pair of LPs,
strong duality holds because~\eqref{eq:dualLPfull},
and therefore also~\eqref{eq:primLPfull}, has a finite
optimum value.
Finally~$\max_{\geq 0} L(\lambda)$ is the same as the optimum
value of~\eqref{eq:dualLPfull} by Lagrangian duality.

It remains to observe that the optimal value of the above
problems is a lower bound to $\OPT$.
We recall that by Lemma~\ref{lem:intObj},
problem~\eqref{eq:minMaxExact}
is a rephrasing of the original interdiction
problem~\eqref{eq:initProb}, and thus also has optimal
value $\OPT$. Finally,~\eqref{eq:minMax} is obtained
from~\eqref{eq:minMaxExact} by relaxation the integrality
condition on $r$. Thus, the optimum value
of~\eqref{eq:minMax}---which is also the optimum
value of~\eqref{eq:minMaxSimple}, \eqref{eq:dualLPfull},
\eqref{eq:primLPfull}, and $\max_{\lambda \geq 0} L(\lambda)$---is
less or equal to $\OPT$, as claimed.
\end{proof}

The following theorem shows that we can efficiently
compute an optimal dual multiplier $\lambda^*$ together
with two integral vectors $r^1, r^2$ that are optimal
solutions to $LP(\lambda^*)$, one of which will turn
out to be a $2$-pseudoapproximation to the considered
interdiction problem~\eqref{eq:initProb}.

\begin{theorem}\label{thm:canComputeRs}
There is an efficient algorithm to compute a maximizer
$\lambda^*$ of $\max_{\lambda \geq 0} L(\lambda)$, and
two vectors $r^1,r^2\in \mathbb{Z}_{\geq 0}^N$ such that:
\begin{enumerate}[(i)]

\item $\exists$ integral $y^1,y^2\in \mathbb{Z}^m$ such
that both $(r^1,y^1)$ and $(r^2,y^2)$ are optimal
solutions to $LP(\lambda^*)$.

\item $c^T r^1 \geq B\geq c^T r^2$.
\end{enumerate}
\end{theorem}

Before proving Theorem~\ref{thm:canComputeRs}, we show
that it implies our main result, Theorem~\ref{thm:2pseudoapprox}.

\begin{theorem}\label{thm:oneRIsGood}
Let $\lambda^*$ be a maximizer of $\max_{\lambda \geq 0} L(\lambda)$,
let $(r^1,y^1), (r^2,y^2)$ be two optimal solutions
to $LP(\lambda^*)$
with $c^T r^1 \geq B\geq c^T r^2$,
and let $\alpha>0$.
Then at least one of the following two conditions holds:
\begin{enumerate}[(i)]
\item\label{item:r1Good} $c^T r^1 \leq (1+\frac{1}{\alpha})B$, or
\item\label{item:r2Good} $b^T y^2 \leq (1+\alpha) L(\lambda^*)$.
\end{enumerate}
Furthermore, if~\eqref{item:r1Good} holds, then $r^1\wedge 1$
is the characteristic vector of a
$(1,1+\frac{1}{\alpha})$-approximation to~\eqref{eq:initProb}.
If~\eqref{item:r2Good} holds, then $r^2\wedge 1$ is the
characteristic vector of a
$(1+\alpha,1)$-approximation to~\eqref{eq:initProb}.
\end{theorem}
\begin{proof}
Before showing that either~\eqref{item:r1Good}
or~\eqref{item:r2Good} holds, we show the second part
of the theorem.

Assume first that~\eqref{item:r1Good} holds.
We recall that problem~\eqref{eq:initProb}
and~\eqref{eq:minMaxExact} are equivalent
due to Lemma~\ref{lem:intObj}. Thus, the objective
value of the interdiction problem~\eqref{eq:initProb}
that corresponds to $r^1\wedge 1$ is given
by $\psi(r^1\wedge 1)$ which, by
Lemma~\ref{lem:dropUpperBound1}, is equal to
$\psi(r^1)$.
Hence, to show that $r^1\wedge 1$ is a
$(1,1+\frac{1}{\alpha})$-approximation, it suffices
to prove $\psi(r^1) \leq L(\lambda^*)$, because
$L(\lambda^*)\leq \OPT$ by Lemma~\ref{lem:optValues}.

Indeed, $\psi(r^1)\leq L(\lambda^*)$ holds due to:
\begin{align*}
L(\lambda^*) &= b^T y^1 - \lambda^*(B - c^T r^1)
  && \text{($(r^1,y^1)$ is a maximizer of $LP(\lambda^*)$)}\\
  &\geq b^T y^1
    && \text{($B\leq c^T r^1$ and $\lambda^*\geq 0$)}\\
  &\geq \psi(r^1)
    && \text{($y^1$ is a feasible solution to the dual of the LP
              defining $\psi(r^1)$)}.
\end{align*}
Similarly, if~\eqref{item:r2Good} holds then the 
objective value corresponding to $r^2\wedge 1$ is
\begin{align*}
\psi(r^2) &\leq b^T y^2
    && \text{($y^2$ is a feasible solution to the dual of the LP
              defining $\psi(r^2)$)}\\
  &\leq (1+\alpha)L(\lambda^*)
    && \text{(by~\eqref{item:r2Good})}\\
  &\leq (1+\alpha)\OPT
    && \text{(by Lemma~\ref{lem:optValues})}.
\end{align*}
Since $r^2$ satisfies $c^T r^2\leq B$, the
characteristic vector $r^2\wedge 1$ is therefore indeed a
$(1+\alpha,1)$-approximation to~\eqref{eq:initProb}.

Hence, it remains to show that at least one
of~\eqref{item:r1Good} and~\eqref{item:r2Good} holds.
Assume for the sake of contradiction that both do not
hold.
Because both $(r^1,y^1)$ and $(r^2,y^2)$ are maximizers
of $LP(\lambda^*)$, also any convex combination of these
solutions is a maximizer. In particular let
$\mu = \frac{\alpha}{1+\alpha}$ and consider
the maximizer $(r_{\mu}, y_{\mu})$ of $LP(\lambda^*)$, where
$r_{\mu} = \mu r^1 + (1-\mu) r^2$ and
$y_{\mu} = \mu y^1 + (1-\mu) y^2$.
We obtain
\begin{align*}
L(\lambda^*) &= b^T y_{\mu} - \lambda^* (B-c^T r_{\mu}) \\
  &\geq (1-\mu) b^T y^2 - \lambda^* (B- \mu c^T r^1)
   && \text{(ignoring $\mu b^T y^1$ and $(1-\mu)\lambda^* c^T r^2$,
which are both $\geq 0$)}\\
  &= \frac{1}{1+\alpha} b^T y^2
     - \lambda^* \left(B - \frac{\alpha}{1+\alpha} c^T r^1\right)
    && \text{(using $\mu=\frac{\alpha}{1+\alpha}$)}\\
  &> L(\lambda^*)
    && \text{(using that both~\eqref{item:r1Good}
              and~\eqref{item:r2Good} do not hold)},
\end{align*}
thus leading to a contradiction and proving the theorem.

\end{proof}

Theorem~\ref{thm:canComputeRs} together with
Theorem~\ref{thm:oneRIsGood} imply our main result,
Theorem~\ref{thm:2pseudoapprox}, due to the following.
Theorem~\ref{thm:canComputeRs} guarantees that we can
compute efficiently $\lambda^*, r^1, r^2$ as needed in
Theorem~\ref{thm:oneRIsGood}. Then, depending whether
condition~\eqref{item:r1Good} or~\eqref{item:r2Good} holds,
we either return $r^1\wedge 1$ or $r^2\wedge 1$ as our
$2$-pseudoapproximation.
Notice that to check whether~\ref{item:r2Good} holds,
we have to compute $L(\lambda^*)$. This can be done efficiently
due to the fact that our description of $\conv(\mathcal{X})$
is box-$w$-DI solvable. More precisely, as already discussed,
$LP(\lambda^*)$ is the dual of~\eqref{eq:primLPfull}
for $\lambda=\lambda^*$
whose optimal value can be computed by box-$w$-DI solvability.
Hence, it remains to prove Theorem~\ref{thm:canComputeRs}.

\subsection{Proof of Theorem~\ref{thm:canComputeRs}}

First we discuss some basic properties of $L(\lambda)$.
We start by observing that $L(\lambda)$ is finite for
any $\lambda >0$. This follows by the fact that
$L(\lambda)$ is the optimal value of~\eqref{eq:primLPfull}
when $\lambda$ is considered fixed. More precisely,
for any fixed $\lambda\geq 0$, the problem~\eqref{eq:primLPfull}
is feasible and bounded. It is feasible because $z=0$
is feasible since $b\geq 0$. Furthermore, it is bounded
since by assumption
$\conv(\mathcal{X})=\{z\in \mathbb{R}^N \mid Az\leq b, z\geq 0\}$
is a polytope.
Additionally, $L(\lambda)$ has the following properties,
which are true for any Lagrangian dual of a finite LP
(see~\cite{boyd_2004_convex} for more details):
\begin{itemize}
\item $L(\lambda)$ is piecewise linear.
\item Let $[\lambda_1, \lambda_2]$ be one of the linear segments
of $L(\lambda)$, let $t\in (\lambda_1,\lambda_2)$, and
$(r_t, y_t)$ be an optimal solution to $L(t)$. Then,
$(r_t, y_t)$ is an optimal solution for the whole
segment, i.e., for any $LP(\lambda)$
with $\lambda\in [\lambda_1,\lambda_2]$.
As a consequence, the slope of the segment is
$c^T r_t - B$.
\end{itemize}
Also, we recall that $L(\lambda)$ can be evaluated efficiently
for any $\lambda \geq 0$; since~\eqref{eq:primLPfull} is box-$w$-DI
solvable, it can be solved for any fixed $\lambda\geq 0$.

We will find an optimal multiplier $\lambda^*\geq 0$ using
bisection. For this, we start by showing two key properties
of $L(\lambda)$. First, we show that any optimal multiplier
$\lambda^*$ to~$L(\lambda)$ is not larger than some upper
bound with polynomial input length. 
Second, we show that each
linear segment of $L(\lambda)$ has some minimal
width, which makes it possible to reach it with a
polynomial number of iterations using bisection.

We recall that
\begin{equation*}
\nu^*=\max\{w^T x \mid Ax \leq b, x\geq 0\}
  = \min \{b^T y \mid A^T y \geq w, y \geq 0\}
\end{equation*}
is the optimal value of the nominal problem without
interdiction, and that $\log(\nu^*)$ is part of the
input size.

\begin{lemma}\label{lem:LSBounded}
If $\lambda^*$ is a maximizer of $L(\lambda)$, 
then $\lambda^*\leq \nu^*$.
Furthermore, for every $\lambda \geq \nu^*$,
$r=0$ is an optimal solution to $LP(\lambda)$.
\end{lemma}
\begin{proof}
Let $r=0\in \mathbb{Z}^N$ and $y^*$ be a
minimizer of $\min\{b^T y \mid A^T y \geq w, y\geq 0\}$.
Hence, in particular, $b^T y^* = \nu^*$.
We first show that for any $\lambda\geq \nu^*$, the
pair $(r,y^*)$ is a minimizer of $LP(\lambda)$.
Assume for the sake of contradiction that there is
some $\lambda \geq \nu^*$ such that $(r,y^*)$
is not a minimizer of $LP(\lambda)$. Let
$(r',y')$ be a minimizer of $LP(\lambda)$ which,
because the dual of $LP(\lambda)$ is box-$w$-DI, can be assumed to be integral.
Clearly, we must have $r'\neq 0=r$, since for
$r=0$, the vector $y^*$ attains by definition
the smallest value in $LP(\lambda)$.
Hence, we obtain
\begin{align*}
b^T y^* - \lambda B &> b^T y' - \lambda B + \lambda c^T r'
  && \text{($(r',y')$ attains a smaller value than $(r,y^*)$ in $LP(\lambda)$)}\\
&\geq -\lambda B + \lambda c^T r'
  && \text{($b^Ty'\geq 0$ since $b\geq 0$ and $y'\geq 0$)},
\end{align*}
which implies
\begin{equation*}
\nu^* > \lambda c^T r'.
\end{equation*}
However, this is a contradiction since $\lambda \geq \nu^*$,
and $c^T r'\geq 1$ because $c\in \mathbb{Z}^N_{> 0}$
and $r'\in \mathbb{Z}^N_{\geq 0}$ is nonzero.
Thus, $(r,y^*)$ is indeed a minimizer of $LP(\lambda)$
for any $\lambda \geq \nu^*$.
However, since $B>0$, this implies
\begin{equation*}
L(\nu^*) = b^T y - \nu^* B >
b^T y - \lambda B = L(\lambda) \qquad \forall \lambda > \nu^*,
\end{equation*}
thus implying the lemma.
\end{proof}

Hence, Lemma~\ref{lem:LSBounded} implies that to find
a maximizer $\lambda^*$ of $L(\lambda)$, we only have
to search within the interval $[0,\nu^*]$.

\begin{lemma}\label{lem:segmentsAreLarge}
Each segment of the piecewise linear function $L(\lambda)$
has width at least $\frac{1}{(c(N))^2}$.
\end{lemma}
\begin{proof}
We start by deriving a property of the kinks of $L(\lambda)$,
namely that they correspond to a rational value $\lambda$
whose denominator is at most $\frac{1}{c(N)}$. Later
we will derive from this property that the distance between
any two kinks is at least $\frac{1}{(c(N))^2}$.

Let $\overline{\lambda} >0$ be the value of a kink of $L(\lambda)$,
i.e., there is one segment of the piecewise linear function
$L(\lambda)$ that ends at $\overline{\lambda}$ and one that
starts at $\overline{\lambda}$. We call the segment ending
at $\overline{\lambda}$ the \emph{left segment} and the one
starting at $\overline{\lambda}$ the \emph{right segment}.
Let $(r^1,y^1)$ be an optimal solution for all $LP(\lambda)$
where $\lambda$ is within the left segment.
Similarly, let $(r^2,y^2)$ be an optimal solution for
the right segment. By box-$w$-DI solvability, we can choose
$(r^i,y^i)$ for $i\in \{1,2\}$ to be integral.
We start by showing that 
$r^1$ and $r^2$ are $\{0,1\}$-vectors,
i.e, $r^1,r^2\in \{0,1\}^N$.
We can rewrite $L(\lambda)$ as follows:
\begin{equation}\label{eq:LAsPsi}
\begin{aligned}
L(\lambda) &= \min\{b^T y - \lambda(B-c^T r)
\mid A^T y + r \geq w, y\geq 0, r\geq 0\}\\
 &=\min_{r\geq 0}\left(-\lambda B + c^T r +
  \max\{(w-r)^T x \mid Ax\leq b, x\geq 0\}\right)\\
 &=\min_{r\geq 0} \left(-\lambda B + c^T r +\psi(r)\right),
\end{aligned}
\end{equation}
where the second equality follows by dualizing the LP
of the first line for a fixed $r\geq 0$, and the
third equality follows by the definition of $\psi$.
By Lemma~\ref{lem:dropUpperBound1}, we have
$\psi(r)=\psi(r\wedge 1)$, and since $c\in \mathbb{Z}_{>0}$,
this implies that a minimizing $r$ is such that
$r=r\wedge 1$. Thus an integral minimizing $r$ satisfies
$r\in \{0,1\}^N$, as desired.

The slope of the left segment is $\beta_1 = -B + c^T r^1$
and the slope of the right segment is
$\beta_2 = -B + c^T r^2$.
Let $\alpha_1 = b^T y^1$
and $\alpha_2 = b^T y^2$.
Again using~\eqref{eq:LAsPsi} we have
$\alpha_i = \psi(r^i)$ for $i\in \{1,2\}$.
This implies that $\alpha_i$ for $i\in \{1,2\}$ is
integral because $\psi(r^i)$ is defined as the optimum
of an LP with integral objective vector over an
integral polytope
$\{x\in \mathbb{R}^N \mid Ax\leq b, x\geq 0\}$.

Because $L(\lambda)$ is concave, the slope decreases
strictly at each kink, i.e., $\beta_1 > \beta_2$.
Furthermore, since the left and right segment
touch at $\overline{\lambda}$, we have
\begin{equation*}
\alpha_1 + \overline{\lambda} \beta_1 =
    \alpha_2 + \overline{\lambda} \beta_2.
\end{equation*}
Because $\beta_1 > \beta_2$ and $\overline{\lambda}>0$,
this implies $\alpha_1 < \alpha_2$, and $\overline{\lambda}$
can be written as
\begin{equation*}
\overline{\lambda}
  = \frac{\alpha_2 - \alpha_1}{\beta_1 - \beta_2}.
\end{equation*}
Notice that
\begin{equation*}
\beta_1 - \beta_2 = c^T (r^1-r^2) \leq \| c \|_1 = c(N),
\end{equation*}
where we use the fact that $r^1, r^2\in \{0,1\}^N$
for the inequality.
In summary, any kink $\overline{\lambda}$ is a rational
number $\frac{p}{q}$ with $p,q\in \mathbb{Z}_{>0}$ and
$q\leq c(N)$.
In particular, this implies that the first segment,
which goes from $\lambda=0$ to the first kink,
has width at least
$\frac{1}{c(N)} \geq \frac{1}{(c(N))^2}$.
The last segment clearly has infinite width. Any
other segment is bordered by two kinks
$\lambda_1 = \frac{p_1}{q_1}$ and
$\lambda_2 = \frac{p_2}{q_2}$ with $\lambda_1 < \lambda_2$
and has therefore a width of 
\begin{align*}
\lambda_1 - \lambda_2 &= \frac{p_1 q_2 - p_2 q_1}{q_1 q_2}\\
  &\geq \frac{1}{q_1 q_2}
    && \text{(since $\lambda_1 -\lambda_2 >0$)}\\
  &\geq \frac{1}{(c(N))^2}
    && \text{(since $q_1,q_2 \leq c(N)$)}.
\end{align*}

\end{proof}

We use the bisection procedure Algorithm~\ref{alg:findLR1R2}
to compute $\lambda^*$,
$r^1$, and $r^2$ as claimed by 
Theorem~\ref{thm:canComputeRs}.
Notice that $L(\lambda^1)$ and $L(\lambda^2)$,
as needed by Algorithm~\ref{alg:findLR1R2} to determine
$\lambda^*$, can be computed due to box-$w$-DI solvability.
Algorithm~\ref{alg:findLR1R2} is clearly efficient;
it remains to show its correctness.

\begin{algorithm2e}[h!]
\DontPrintSemicolon
Initialization: $\lambda^1=0$, $\lambda^2=\nu^*$,
$r^1 = \chi^N$, $r^2=0$\;

\For{$i=1,\dots, 1+\lfloor\log_2(\nu^* (c(N))^2) \rfloor$}{
  $\lambda = \frac{1}{2}(\lambda^1+\lambda^2)$.\;
\medskip

  Use box-$w$-DI solvability oracle to compute
  integral $r\in \mathbb{Z}_{\geq 0}$ satisfying that
  there is a $y$ such that $(r, y)$ is an optimal
  solution to $LP(\lambda)$.\;

\medskip

\uIf{$-B + c^T r\geq 0$}{
  $\lambda^1 = \lambda$.\\
  $r^1 = r$.
}
\Else{
  $\lambda^2 = \lambda$.\\
  $r^2 = r$.
}

}

Compute $\lambda^*$ as the intersection of the
two segments at $\lambda^1$ and
$\lambda^2$:
\begin{equation*}
\lambda^* = \frac{L(\lambda^2) - L(\lambda^1)
 - \lambda^2 (-B+c^T r^2) + \lambda^1 (-B + c^T r^1)}
{c^T (r^1 - r^2)}.
\end{equation*}

\medskip

\Return{$\lambda^*, r^1, r^2$.}

\caption{Computing $\lambda^*$, $r^1$ and $r^2$ as claimed
by Theorem~\ref{thm:canComputeRs}}
\label{alg:findLR1R2}

\end{algorithm2e}

\begin{lemma}
$\lambda^*$, $r^1$ and $r^2$ as returned by
Algorithm~\ref{alg:findLR1R2} fulfill the properties
required by Theorem~\ref{thm:canComputeRs}.
\end{lemma}
\begin{proof}
  Notice that throughout the algorithm
  the following
invariant is maintained: $r^i$ is an optimal
solution to $LP(\lambda^i)$ for $i\in \{1,2\}$.
Furthermore, $-B+c^T r^1 \geq 0$ and $-B+c^T r^2 < 0$.
We highlight that after initialization, these two
invariants are maintained because
$-B + c^T \chi^N = -B + c(N) >0$ because we
assumed $B < c(N)$ to avoid the trivial special
case when everything is interdicted.
Additionally, $-B + c^T 0 = -B < 0$.
Also note that $r^2=0$ is an optimal solution
to $LP(\nu^*)$ by Lemma~\ref{lem:LSBounded}.

Due to this invariant and the fact that $L(\lambda)$
is concave, we know that there is a maximizer
$\lambda^*$ of $L(\lambda)$ within $[\lambda^1, \lambda^2)$.

Observe that the distance $\lambda^2 - \lambda^1$ halves
at every iteration of the for loop.
Consider now $\lambda^1$ and $\lambda^2$ after
the for loop. Their distance is bounded by
\begin{align*}
\lambda^2 - \lambda^1 &=
  \nu^* \left(\frac{1}{2}\right)^{%
     1+\lfloor \log_2(\nu^* (c(N))^2)\rfloor}
 < \nu^* \left(\frac{1}{2}\right)^{\log_2(\nu^*(c(N))^2)}
 = \frac{1}{(c(N))^2}.
\end{align*}

Hence, the distance between $\lambda^2$ and $\lambda^1$ is
less then the width of any segment of the piecewise linear
function $L(\lambda)$, due to Lemma~\ref{lem:segmentsAreLarge}.
This leaves the following options.
Either one of $\lambda^1$ or $\lambda^2$ is a maximizer
of $L(\lambda)$, and the other one is in the interior
of the segment to the left or right, respectively.
Or, neither $\lambda^1$ nor $\lambda^2$ is a maximizer
of $L(\lambda)$. In this case $\lambda^1$ and $\lambda^2$
are in the interior of the segment to the left and
right, respectively, of the unique maximizer $\lambda^*$.
In all of these cases, the solutions $r^1$ and $r^2$ are
both optimal with respect to some maximizer $\lambda^*$ of
$L(\lambda)$, since they are on two segments that
meet on an optimal multiplier~$\lambda^*$.

It remains to prove that the returned $\lambda^*$ is
correct.
Since both $r^1$ and $r^2$ are optimal solutions
to $L(\lambda^*)$ for some maximizer $\lambda^*$,
we have
\begin{equation}\label{eq:bothRiOpt}
L(\lambda^*) = b^T y^1 - \lambda^*(B-c^T r^1)
             = b^T y^2 - \lambda^*(B-c^T r^2).
\end{equation}
Furthermore,
\begin{align*}
L(\lambda^1) &= b^T y^1 - \lambda^1 (B-c^T r^1), \text{ and} \\
L(\lambda^2) &= b^T y^2 - \lambda^2 (B-c^T r^2).
\end{align*}
By replacing $b^T y^i= L(\lambda^i) + \lambda^i (B-c^T r^i)$,
for $i\in \{1,2\}$, in~\eqref{eq:bothRiOpt} and
solving for $\lambda^*$, we obtain
\begin{equation*}
\lambda^* = \frac{L(\lambda^2) - L(\lambda^1)
 - \lambda^2 (-B+c^T r^2) + \lambda^1 (-B + c^T r^1)}
{c^T (r^1 - r^2)},
\end{equation*}
thus showing that the returned $\lambda^*$ is
indeed optimal.
\end{proof}

Hence, even the somewhat limited access through
box-$w$-DI solvability that we assume
to our optimization problem is enough to obtain
an efficient $2$-pseudoapproximation for the
interdiction problem due to the efficiency
of the bisection method described in
Algorithm~\ref{alg:findLR1R2}.
However, in many concrete settings, more efficient
methods can be employed to get an optimal multiplier
$\lambda^*$ and optimal integral dual solutions
$r^1, r^2$. In particular, often one can even obtain
strongly polynomial procedure by employing Megiddo's
parametric search
technique~\cite{megiddo_1979_combinatorial}.
We refer the interested reader
to~\cite{juttner2006budgeted} for a technical details
of how this can be done in a very similar context.

\section{Matroids: weighted case and submodular costs}
\label{sec:matroidInt}

In this section we consider the problem of interdicting
a feasible set $\mathcal{X}\subseteq \{0,1\}^N$ that
corresponds to the independent sets of a matroid.
It turns out that we can exploit structural properties
of matroids to solve natural generalization of the
interdiction problem considered in 
Theorem~\ref{thm:2pseudoapprox}.
In particular, even for arbitrary nonnegative
weight functions $w\in \mathbb{Z}_{\geq 0}^N$, we can
obtain a $2$-pseudoapproximation for the corresponding
interdiction problem.
What's more is that we can achieve this when the
interdiction costs are submodular, rather than just linear.

For clarity, we first discuss in
Section~\ref{sec:weightedCase} a technique to reduce
arbitrary nonnegative weights to the case of
$\{0,1\}$-objectives that was mentioned previously.
In Section~\ref{sec:submCosts}, we then build up and
extend this technique to also deal with submodular
interdiction costs.

\subsection{Weighted case}
\label{sec:weightedCase}

Let $M=(N,\mathcal{I})$ be a matroid, and let
$w:N \rightarrow \mathbb{Z}_{\geq 0}$. The canonical
problem we want to interdict is the problem of
finding a maximum weight independent set, i.e.,
$\max\{w(I) \mid I\in \mathcal{I}\}$.
Let $r_w:2^N \rightarrow \mathbb{Z}_{\geq 0}$ be the
\emph{weighted rank} function, i.e.,
\begin{equation*}
r_w(S) = \max\{w(I) \mid I\subseteq S, I\in \mathcal{I}\}.
\end{equation*}
In words, $r_w(S)$ is the weight of a heaviest independent
set that is contained in $S$.
We recall a basic fact on weighted rank
functions~\cite[Section~44.1a]{schrijver2003combinatorial}.

One key observation we exploit is that the maximum
weight independent set can be rephrased as maximizing
an all-ones objective function over the following
polymatroid:
\begin{equation}\label{eq:descPw}
P_w = \{x\in \mathbb{R}^N_{\geq 0} \mid x(S) \leq r_w(S)
\;\;\forall S\subseteq N\}.
\end{equation}
Even more importantly, we do not only have 
$\max\{ x(N) \mid x\in P_w\} =
\max\{w(I) \mid I\in \mathcal{I}\}$,
but we also have that the problem of interdicting
the maximum weight independent set problem of a matroid
maps to the problem of interdicting the corresponding
all-ones maximization problem on the polymatroid.
This is formalized through the lemma below.
\begin{lemma}\label{lem:toPolymat}
For any $R \subseteq N$, we have
\begin{equation*}
\max\{x(N) \mid x \in P_w, x(R)=0\} =
\max\{w(I) \mid I\in \mathcal{I}, I\subseteq N\setminus R\}.
\end{equation*}
\end{lemma}
\begin{proof}
Observe that the right-hand side of the above
equality is, by definition, equal to $r_w(N\setminus R)$.

\textbf{lhs $\leq$ rhs:}
Let $x^*$ be a maximizer of
$\max\{x(N) \mid x\in P_w, x(R)=0\}$. We have
\begin{align*}
x^*(N) &= x^*(N\setminus R) & \text{($x^*(R)=0$)}\\
        &\leq r_w(N\setminus R) &\text{(since $x^*\in P_w$)},
\end{align*}
thus showing the desired inequality.

\textbf{lhs $\geq$ rhs:}
Conversely, let $I^*$ be a maximizer of
$\max\{w(I) \mid I\in \mathcal{I}, I\subseteq N\setminus R\}$.
Hence, $w(I^*)=r_w(N\setminus R)$.
Define $y\in \mathbb{R}^N_{\geq 0}$ by
\begin{equation*}
y(e) = \begin{cases}
  w(e) & \text{if } e\in I^*\\
     0 & \text{if } e\in N\setminus I^*.
\end{cases}
\end{equation*}

Clearly, $y(N) = w(I^*)=r_w(N\setminus R)$.
Thus, to show that the left-hand side of the equality
of Lemma~\ref{lem:toPolymat} is at least as large as
the right-hand side, it suffices to show that $y$ is
feasible to the maximization problem on the left-hand
side, i.e., $y(R)=0$ and $y\in P_w$.
We have $y(R)=0$ since $I^*\subseteq N\setminus R$.
Furthermore,
\begin{align*}
y(S) = w(S\cap I^*) \leq r_w(S)
  \qquad \forall S\subseteq N,
\end{align*}
where the inequality follows from $S\cap I^*\in \mathcal{I}$.
Hence, this implies $y\in P_w$ and completes the proof.
\end{proof}

We therefore can focus on the problem
$\max\{x(N) \mid x\in P_w, \;x(R) =0\}$ to which we
can now apply Theorem~\ref{thm:2pseudoapprox}.
For this it remains to observe that $P_w$ is $1$-down-closed because
it is down-closed. Furthermore, the description of $P_w$
given by~\eqref{eq:descPw} is box-$1$-DI solvable since
it is well-known to be even box-TDI, a property that
holds for all
polymatroids~\cite[Section~44.3]{schrijver2003combinatorial}, and one can efficiently find
an optimal integral dual 
solution to the problem of finding a maximum size point
over~\eqref{eq:descPw} with upper box constraints.
In fact, this problem can be interpreted as a maximum
cardinality polymatroid intersection problem,
one polymatroid being $P_w$ and the other one being
defined by the upper box constraints.
An optimal integral dual solution to the maximum
cardinality polymatroid intersection problem
can be found in strongly
polynomial time by standard techniques
(for clarity we provide some more details about
this in Section~\ref{sec:submCosts}).
In summary, our technique presented
in Section~\ref{sec:2pseudoapprox}
to obtain $2$-pseudoapproximations therefore indeed applies
to this setting.

\subsection{Submodular costs}
\label{sec:submCosts}

In this section, we show how to obtain a
$2$-pseudoapproximation
for the interdiction of the maximum weight independent set
of a matroid with submodular interdiction costs.
When dealing with submodular interdiction costs, we assume
that the interdiction costs $\kappa$ are a nonnegative
and monotone submodular function
$\kappa:2^N \rightarrow \mathbb{R}_{\geq 0}$.
As before, a removal set $R\subseteq N$ has to satisfy the
budget constraint, i.e., $\kappa(R)\leq B$.
We assume that the submodular function $\kappa$ is given by
a value oracle.

To design a $2$-pseudoapproximation, we will describe a way
to formulate the problem such that it can be attacked with
essentially the same techniques as described in
Section~\ref{sec:2pseudoapprox}.
For simplicity of presentation, and to avoid replicating
reasonings introduced in Section~\ref{sec:2pseudoapprox},
we focus on the key differences in this section, and refer
to Section~\ref{sec:2pseudoapprox} for proofs that are
essentially identical.

We extend the model for the weighted case. A variable
$q(S)$ is introduced for each set $S\subseteq N$.
In the non-relaxed mathematical program, we have
$q\in \{0,1\}^{2^N}$, and only one variable $q(S)$
is equal to one, which indicates the set $S$ of elements
we interdict.
Below is a mathematical description of a relaxation,
where we allow the variables $q(S)$ to take real
values. If instead of allowing $q(S)\in \mathbb{R}_{\geq 0}$,
we set $q(S)\in\{0,1\}$, then the mathematical program below
would be an exact description of the interdiction problem
with submodular interdiction costs.

\begin{equation}
\begin{array}{>{\displaystyle}rr@{\;\;}c@{\;\;}ll}
\min_{q\in \mathbb{R}^{2^N}} \max_{x \in \mathbb{R}^n} &
   (1-\sum_{S\subseteq N} \chi^S \cdot q(S))^T x & & &\\
              &x(S)      &\leq &r_w(S) & \forall S\subseteq N \\
              &x        &\geq &0  & \\
              &   \sum_{S\subseteq N} \kappa(S)\cdot q(S) &\leq &B  & \\
              &   \sum_{S\subseteq N} q(S)     &\leq &1  &\\
              &   q(S)     &\geq &0  & \forall S\subseteq N\\
\end{array}
\label{eq:minMaxSubm}
\end{equation}

We start by dropping the constraint
$\sum_{S\subseteq N} q(S) \leq 1$. As we will see later,
this does not change the objective value.
This step
is similar to dropping the constraint $r\leq 1$ when going
from~\eqref{eq:minMax} to~\eqref{eq:minMaxSimple} in the
standard setting of our framework without submodular
interdiction costs.
We thus obtain the following mathematical program.

\begin{equation}
\begin{array}{>{\displaystyle}rr@{\;\;}c@{\;\;}ll}
\min_{q\in \mathbb{R}^{2^N}} \max_{x \in \mathbb{R}^n} &
   (1-\sum_{S\subseteq N} \chi^S \cdot q(S))^T x & & &\\
              &x(S)      &\leq &r_w(S) & \forall S\subseteq N \\
              &x        &\geq &0  & \\
              &   \sum_{S\subseteq N} \kappa(S)\cdot q(S) &\leq &B  & \\
              &   q(S)     &\geq &0  & \forall S\subseteq N\\
\end{array}
\label{eq:minMaxSimplSubm}
\end{equation}

Now, by dualizing the inner problem we get the following LP.

\begin{equation}
\begin{array}{>{\displaystyle}rr@{\;\;}c@{\;\;}ll}
\min & \sum_{S\subseteq N} r_w(S) y(S) & & &\\
              &\left(\sum_{S\subseteq N: e\in S} y(S)\right)
          + \left( \sum_{S\subseteq N: e\in S} q(S) \right)
                   &\geq & 1 & \forall e\in N \\
              &   y(S)     &\geq &0  & \forall S\subseteq N\\
              &   q(S)     &\geq &0  & \forall S\subseteq N\\
              &   \sum_{S\subseteq N} \kappa(S)\cdot q(S) &\leq &B  & \\
\end{array}
\label{eq:dualLPfullSubm}
\end{equation}

As in the case with linear interdiction costs, we dualize the
budget constraint with a Lagrangian multiplier $\lambda$
to obtain the following family of LPs, parameterized by
$\lambda$:

\begin{equation}
\begin{array}{>{\displaystyle}rr@{\;\;}c@{\;\;}ll}
L(\lambda) = \min & \sum_{S\subseteq N} r_w(S) y(S)
               +\lambda\left( \sum_{S\subseteq N} \kappa(S) \cdot q(S)\right)
               - \lambda B
              & & &\\
              &\left(\sum_{S\subseteq N: e\in S} y(S)\right)
          + \left( \sum_{S\subseteq N: e\in S} q(S) \right)
                   &\geq & 1 & \forall e\in N \\
              &   y(S)     &\geq &0  & \forall S\subseteq N\\
              &   q(S)     &\geq &0  & \forall S\subseteq N\\
\end{array}
\tag{LP($\lambda$)}
\label{eq:dualLPlagS}
\end{equation}

It remains to observe that for any
$\lambda \geq 0$,~\ref{eq:dualLPlagS}
is the dual of a maximum cardinality
polymatroid intersection problem---when
forgetting about the constant term
$-\lambda B$---where the two
polymatroids are defined by the submodular
functions $r_w$ and $\lambda\cdot \kappa$, respectively.
A key result in this context is that there is
a set $A\subseteq N$ such that the optimal primal
value, which is equal to the optimal dual value by
strong duality, is equal to $\lambda \kappa(A) + r_w(N\setminus A)$
(see~\cite[Section~46.2]{schrijver2003combinatorial}).
This implies that defining $q(A)=1$, $y(N\setminus A)=1$,
and setting all other entries of $q$ and $y$ to zero
is an optimal solution to~\eqref{eq:dualLPlagS}.
Furthermore, such a set $A$ can be found in strongly
polynomial time~\cite[Section~47.1]{schrijver2003combinatorial}.
Note that this fact also implies that dropping the
constraint $\sum_{S\subseteq N} q(S) \leq 1$ when going
from~\eqref{eq:minMaxSubm} to~\eqref{eq:minMaxSimplSubm}
did not change the objective value of the mathematical
program.
Furthermore, we can evaluate $L(\lambda)$ efficiently for
any $\lambda\geq 0$.

From this point on, the approach is identical to
the one presented in Section~\ref{sec:2pseudoapprox}
for linear interdiction costs.
More precisely, we 
determine the optimal dual multiplier $\lambda^*$
and two optimal dual solutions $(q^1,y^1)$, $(q^2, y^2)$
to $LP(\lambda^*)$ such that 
\begin{enumerate}[(i)]
\item The dual solutions have the above-mentioned property that
all four vectors $y^1, q^1, y^2, q^2$ only have $0$-entries with
the exception of a single $1$-entry. Let $R_1, R_2\subseteq N$
be the sets such that $q^1(R^1)=q^2(R^2)=1$.

\item One solution has interdiction cost that is upper bounded
by the budget and one has an interdiction cost that is lower
bounded by the budget, i.e., 
$\kappa (A^1) \leq B \leq \kappa (A^2)$.
\end{enumerate}

The value $\lambda^*$ and vectors $q^1, y^1, q^2, y^2$ can
either be found by bisection, as described in
Section~\ref{sec:2pseudoapprox}, or they can be obtained
in strongly polynomial time via Megiddo's parametric
search technique (see~\cite{juttner2006budgeted} for details).
An identical reasoning as used in Theorem~\ref{thm:oneRIsGood}
shows that one of $R^1$ or $R^2$ is a $2$-pseudoapproximation.

\section{Refinements for bipartite $b$-stable set interdiction}
\label{sec:bIndep}

This section specializes our approach to the interdiction
of $b$-stable sets in a bipartite graph.
We recall that given is a bipartite graph
$G=(V,E)$ with bipartition $V=I\cup J$ and edge
capacities~$b\in\mathbb{Z}_{\geq 0}^E$.
A $b$-stable set is a vector $x\in \mathbb{Z}_{\geq 0}^N$
such that $x(i)+x(j)\leq b(\{i,j\})$ for each $\{i,j\}\in E$.
The \emph{value} of a $b$-stable set $x$ is given
by $x(V)$. The maximum $b$-stable set problem asks
to find a $b$-stable set of maximum value.
Furthermore, we are given an interdiction cost
$c:V\rightarrow \mathbb{Z}_{>0}$ for each vertex, and an
interdiction budget $B\in \mathbb{Z}_{>0}$.
As usual, the task is to remove a subset $R\subseteq V$
with $c(R)\leq B$ such that value of a maximum $b$-stable
set in the graph obtained from $G$ by removing $R$ is as
small as possible.

In Section~\ref{sec:PTASBIndepSet} we show how our approach
can be adapted to get a PTAS for $b$-stable set interdiction,
thus proving Theorem~\ref{thm:PTASBIndepSet}.
In Section~\ref{sec:maxCardBIndepSet} we complete the
discussion on $b$-stable set interdiction by presenting
an exact algorithm to solve the interdiction problem of
the classical stable set problem in bipartite graphs,
which corresponds to the case when $b$ is the all-ones vector.

Before presenting these results,
we remark that $b$-stable set problem has
also a well-known vertex-capacitated variant. 
In this case an additional vector
$u\in\mathbb{Z}_{\geq 0}^V$ is given and
constraints $x\leq u$ are imposed.
The vertex-capacitated problem can easily be reduced
to the uncapacitated problem by adding two
additional vertices $v_I, v_J$, where $v_I$ is added to $I$
and $v_J$ to $J$, and connecting $v_I$ to all vertices in $J$
and $v_J$ to all vertices in $I$.
Finally, by choosing $b(\{v_I,j\})=u(j)$ for $j\in J$ and
$b(\{v_J, i\})=u(i)$ for $i\in I$, one obtains a $b$-stable
set problem that is equivalent to the vertex-capacitated version.
Furthermore, a vertex interdiction strategy for minimizing
the maximum $b$-independent set problem in this auxiliary graph
carries over exactly to the vertex-capacitated variant.
Thus, the approach we present can also deal with vertex
capacities.

\subsection{PTAS by exploiting adjacency structure}
\label{sec:PTASBIndepSet}

As in our general approach, we start with the
relaxation~\eqref{eq:minMax}.
Below, we adapt the description of the relaxation to
this specialized setting highlight some structural aspects of the problem.
\begin{equation}
\begin{array}{>{\displaystyle}rr@{\;\;}c@{\;\;}l}
\min_{r\in \mathbb{R}^V} \max_{x \in \mathbb{R}^V} &(1-r)^T x & & \\
              &A x      &\leq &b \\
              &x        &\geq &0 \\
              &   c^T r &\leq &B \\
              &0 \;\; \leq \;\; r        &\geq &1 \\
\end{array}
\label{eq:minMaxBIndepSet}
\end{equation}
Notice that the matrix $A\in \{0,1\}^{E\times V}$ is
the incidence matrix of the bipartite graph $G$, i.e.,
$A(e,v)=1$ if and only if $v\in V$ is one of the endpoints
of $e\in E$. This matrix is well known to be
totally unimodular (TU)~\cite{korte_2012_combinatorial}.
Similar to our general approach, we could now drop
the constraint $r\leq 1$. However, since this does not
lead to a further simplification in this setting, we will
keep this constraint.
Following our general approach, we dualize the inner
maximization problem to obtain the following linear program.
\begin{equation}
\begin{array}{>{\displaystyle}rr@{\;}r@{\;\;}c@{\;\;}l}
\min          &b^T y  &     &     &  \\
              &A^T y  &+\;r &\geq &1 \\
              &y   &        &\geq &0 \\
              &    & c^T r  &\leq &B \\
              &  0\;\;  &  \leq\;\; r  &\leq &1 \\
\end{array}
\label{eq:dualLPfullBIndepSet}
\end{equation}

Observe that $\{0,1\}$-solutions
to~\eqref{eq:dualLPfullBIndepSet} have a nice combinatorial
interpretation. More precisely, they correspond to a
subset $R\subseteq V$ of the vertices (where $\chi^R = r$)
with $c(R)\leq B$ and an edge
set $F\subseteq E$ (where $\chi^F = y$) such that
$F$ is an edge cover in the graph obtained from $G$
by removing the vertices $R$.

Not surprisingly, apart from
the budget constraint $c^T r \leq B$,
the feasible region of the above LP closely resembles
the bipartite edge cover polytope. We will make this
link more explicit in the following with the goal to
exploit well-known adjacency properties of the bipartite
edge cover polytope.
First, notice that for any feasible solution $(y,r)$
to~\eqref{eq:dualLPfull}, the vector $(y\wedge 1, r)$
is also feasible with equal or lower objective value.
This follows from the fact that $A$ is a $\{0,1\}$-matrix.
Hence, we can add the constraint $y\leq 1$ without changing
the problem to obtain the following LP.
\begin{equation}
\begin{array}{>{\displaystyle}rr@{\;}r@{\;\;}c@{\;\;}l}
\min          &b^T y  &     &     &  \\
              &A^T y  &+\;r &\geq &1 \\
              &    & c^T r  &\leq &B \\
              &  0\;\;  &  \leq\;\; y  &\leq &1 \\
              &  0\;\;  &  \leq\;\; r  &\leq &1 \\
\end{array}
\label{eq:dual2LPfullBIndepSet}
\end{equation}
The feasible region of the above
LP is given by intersection the polyope
\begin{equation*}
P=\left\{
\begin{pmatrix}
y\\
r
\end{pmatrix}
\in \mathbb{R}^{|E|+|V|}
\;\middle\vert\;
A^T y + r \geq 1, 0\leq y \leq 1, 0\leq r\leq 1
\right\}
\end{equation*}
with the half-space
$\{(y,r)\in \mathbb{R}^{|E|+|V|} \mid c^T r\leq B\}$.
Notice that $P$ is integral because the matrix
$A$ is TU.
The key property we exploit is that $P$
has very well-structured adjacency properties,
because it can be interpreted as a face of a bipartite 
edge cover polytope, a polytope whose adjacency
structure is well known.
More precisely, it turns out that any two adjacent vertices
of $P$ represent solutions that do
not differ much in terms of cost and objective function.
Hence, similar to our general approach, we compute
two vertex solutions of $P$, one over budget but with
a good objective value and the other one under budget,
with the additional property that they are adjacent
on $P$. We then return the one solution that is
budget-feasible. 
This procedure as stated does not yet lead
to a PTAS, but it can be transformed into one
by a classical preprocessing technique that we
will briefly mention at the end.

We start by introducing a bipartite edge cover polytope $P'$
such that $P$ is a face of $P'$. To simplify the exposition,
we do a slight change to the above sketch of the algorithm.
More precisely, we will restate~\eqref{eq:dual2LPfullBIndepSet}
in terms of a problem on $P'$ and then work on the polytope
$P'$ instead of $P$.
We will define $P'$ with a system of linear constraints.
It has two new rows and one new variable $r_{IJ}$ in addition to the constraints~$A^Ty+r\geq 1$ of $P$.
The rows correspond to two new vertices in the graph, one in $I$ and one in $J$, and the new variable
is for an edge between the two new vertices.
The updated constraints are
\begin{equation}\label{eq:auxConstrBIS}
\underbrace{\left(\begin{array}{ccc} 
A^T & I   & 0\\
0 & (\chi^J)^T & 1\\
0 & (\chi^I)^T & 1
\end{array}\right)}_{D:=}
\left(\begin{array}{c}
y\\
r\\
r_{\textsc{IJ}}
\end{array}\right)\geq 1,
\end{equation}
where $\chi^I, \chi^J \in \{0,1\}^V$ are the characteristic
vectors of $I\subseteq V$ and $J\subseteq V$, respectively.
Let $D$ be the $\{0,1\}$-matrix on the left-hand side
of the constraint~\eqref{eq:auxConstrBIS}.
Notice that $D$ is the vertex-edge incidence matrix
of a bipartite graph $G'=(V',E')$, where $G'$ is
obtained from $G$ as follows:
add one new vertex $w_I$ to $I$ and one new
vertex $w_J$ to $J$; then connect $w_I$ to
all vertices in $J\cup\{w_J\}$ and $w_J$ to
all vertices in $I$.
Hence, $I'=I\cup\{w_I\}$ and $J'=J\cup\{w_J\}$
is a bipartition of $V'$.
Since $D$ is an incidence matrix of a bipartite
graph, it is TU.
For easier reference to the different types of
edges in $G'$ we partition $E'$ into the
edge $E$, the edge set
\begin{equation*}
E_R = \{\{w_I, j\} \mid j\in J\} \cup
      \{\{i,w_J\} \mid i\in I\},
\end{equation*}
and the single edge $f=\{w_I, w_J\}$, i.e.,
$E' = E \cup E_R \cup \{f\}$.

Now consider the edge cover polytope that
corresponds to $D$:
\begin{equation*}
P' = \left\{
\begin{pmatrix}
y\\
r\\
r_{IJ}
\end{pmatrix}
\in \mathbb{R}^{|E|+|V|+1}) 
\;\middle\vert\; D \cdot
\begin{pmatrix}
y\\
r\\
r_{IJ}
\end{pmatrix}
\geq 1, 0\leq y,r,r_{IJ} \leq 1\right\}.
\end{equation*}
Notice that $P$ is obtained from $P'$ by
considering the face of $P'$ defined by
$r_{IJ}=1$, and projecting out the variable
$r_{IJ}$.
Every vertex $y,r,r_{IJ}$ of $P'$
is a characteristic vector
of an edge cover in $G'$, where $y$ represents
the characteristic vector of the edges in $E$,
the vector $r$ is the characteristic vector
of the edges in $E_R$,
and $r_{IJ}=1$ indicates that $f$ is part
of the edge cover.

We can now restate~\eqref{eq:dual2LPfullBIndepSet}
as follows in terms of $P'$:
\begin{equation}\label{eq:budgetedEdgeCover}
\min\left\{b^T y \;\middle\vert\;
\begin{pmatrix}
y\\
r\\
r_{IJ}
\end{pmatrix}
\in P'
,c^T r \leq B\right\}.
\end{equation}
Indeed, one can always choose for free $r_{IJ}=1$
in the above LP, since $r_{IJ}$ does not appear in
the objective.
Furthermore, when setting $r_{IJ}=1$,
the LP~\eqref{eq:budgetedEdgeCover} has the
same feasible vectors $(y,r)$
as~\eqref{eq:dual2LPfullBIndepSet}.
We can thus focus on~\eqref{eq:budgetedEdgeCover}
instead of~\eqref{eq:dual2LPfullBIndepSet}.

One can interpret an edge cover $F$ in $G'$ as an
interdiction strategy of the original problem as follows.
Every vertex $v\in V$ that is incident with
either $w_I$ or $w_J$ through an edge of $F$ will be
interdicted.
To obtain a better combinatorial interpretation
of~\ref{eq:budgetedEdgeCover}, we extend the
vectors $b$ and $c$ to all edges $E'$.
More precisely, $b$ is only defined for edges
in $E$. We set $b(e)=0$ for $e\in E'\setminus E$.
Furthermore, the vector $c$ can be interpreted
as a vector on the edges $E_R$, where
$c(\{w_I, j\}):=c(j)$ and $c(\{i,w_J\}):=c(i)$
for $i\in I$ and $j\in J$. For $e\in E'\cup\{f\}$
we set $c(e)=0$.
Using this notation, the best $\{0,1\}$-solution
to~\eqref{eq:budgetedEdgeCover} can be interpreted
as an edge cover $F$ of $G'$ that minimizes $b(F)$
under the constraint $c(F)\leq B$.
One can observe that the best $\{0,1\}$-solution
to~\eqref{eq:budgetedEdgeCover} corresponds to an
optimal interdiction set for the original non-relaxed
interdiction problem.

Also, we want to highlight that problems of this type,
where a combinatorial optimization problem has to
be solved under an additional linear packing
constraint with nonnegative coefficients are also
known as \emph{budgeted optimization problems} or
\emph{restricted optimization problems}
and have been studied for various problem settings,
like spanning trees, matchings, and shortest paths
(see~\cite{grandoni_2014_new} and references
therein for more details).
The way we adapt our procedure to exploit adjacency
properties of the edge cover polytope is inspired by
procedures to find budgeted matchings
and spanning trees~\cite{ravi_1996_constrained,berger2008budgeted,grandoni_2014_new}.

We compute an optimal vertex solution
$p^* = (y^*,r^*,r^*_{IJ}=1)$
to~\eqref{eq:budgetedEdgeCover}
via standard linear programming techniques.
If $r^*$ is integral, i.e., $r^*\in \{0,1\}^V$, then
$r^*$ corresponds to an optimal interdiction set since
it is optimal for the relaxation and integral.
Hence, assume $r^*$ not to be integral from now on.
This implies that $p^*$ is in the
interior of an edge of $P'$, since it is a vertex
of the polytope obtained by intersecting $P'$ 
with a single additional constraint.
This edge of the polytope $P'$ is described by looking
at the constraints of $P'$ that are tight with respect
to the optimal vertex solution. From this description
of the edge, we can efficiently compute its two endpoints
$y^1, r^1, r_{IJ}^1=1$ and $y^2,r^2, r_{IJ}^2=1$, which 
are vertices of $P'$ and therefore integral.
These two solutions correspond to edge covers
$F^1, F^2\subseteq E'$ in $G'$
with $f\in F^1\cap F^2$.
For simplicity, we continue to work with these
edge covers $F^1$ and $F^2$.
One of these edge covers will violate the budget
constraint and be superoptimal,
say the first one, i.e., $c(F^1) > B$ and $b(F^1) < b^T y^*$,
and the other one strictly satisfies the budget constraint
 and is suboptimal,
i.e., $c(F^2) < B$ and $b(F^2) > b^T y^*$.
Hence, this is just a particular way to obtain
two solutions as required by our general approach,
with the additional property that they are
adjacent on the polytope $P'$.

The key observation is that $F^2$ is not just
budget-feasible, but almost optimal. We prove
this by exploiting the following adjacency property
of edge cover polytopes shown by Hurkens.

\begin{lemma}[Hurkens~\cite{hurkens1991diameter}]
\label{lem:hurkens}
Two edge covers $U_1$ and $U_2$ of a bipartite graph
are adjacent if and only if $U_1\Delta U_2$ is an
alternating cycle or an alternating path with endpoints
in $V(U_1\cap U_2)$, where $V(U_1\cap U_2)$ denotes all
endpoints of the edges in $U_1\cap U_2$.
\end{lemma}

\begin{lemma}
$b(F^2) \leq b^T y^* + 2 b_{\max}$, where
$b_{\max} = \max_{e\in E} b(e)$.
\end{lemma}
\begin{proof}
We will prove the statement by constructing
a new edge cover $Z\subseteq E'$ of $G'$
with the following two properties:
\begin{enumerate}[(i)]
\item\label{item:zCostOk} $c(Z) \leq c(F^2)$, and
\item\label{item:zBOk} $b(Z) \leq b(F^1)+2b_{\max}$.
\end{enumerate}
We claim that this implies the result due
to the following.
First observe that there can be no edge cover
$W$ of $G'$ such that $c(W)\leq c(F_2)$
and $b(W) < b(F_2)$.
If such an edge cover existed, then 
$p^*$ would not be an
optimal solution to~\eqref{eq:budgetedEdgeCover},
because $p^*$ is a convex combination
of $\chi^{F^1}$ and $\chi^{F^2}$, and
by replacing $F^2$ by $W$ one would obtain
a new budget-feasible solution with
lower objective value.
Hence, if \eqref{item:zCostOk} then $b(Z)\geq b(F^2)$,
which in turn implies
\[b(F^2) \leq b(Z)
\overset{\eqref{item:zBOk}}{\leq} b(F^1)
+ 2 b_{\max} \leq b^T y^* + 2 b_{\max}.
\]
Hence, it remains to prove the existence
of an edge cover $Z\subseteq E'$
satisfying~\eqref{item:zCostOk}
and~\eqref{item:zBOk}.%

By Lemma~\ref{lem:hurkens},
$U= F^1\Delta F^2$ is either an
alternating path or cycle.
In both cases, $U$ contains at most $4$ edges
of $E_R$, at most $2$ in $E_R\cap F^1$
and at most $2$ in $E_R\cap F^2$.
Let $E_R^1 = E_R\cap U \cap F^1$
be the up to two edges of $U$ in $E_R\cap F^1$.
Consider $X=F^1\setminus E_R^1$. 
$X$ is not necessarily an edge cover because
we removed up to two edges of $E_R$.
Hence, there may be up to $4$ vertices not
covered by $X$. However, the up to two edges of
$E_R$ that we removed to obtain $X$ from $F^1$
are both incident with one of the two vertices
$w_I$ and $w_J$.
Since $f\in X$ because $f\in F^1$,
the two vertices $w_I$, $w_J$
remain covered by $X$. Hence, there are at most two
vertices $i,j\in V$ that are not covered by
$X$. These two vertices are covered by
the edge cover $F^2$. Thus, there are up to
two edges $g,h\in F^2\setminus F^1$ that
touch $i$ and $j$. Now consider the
edge cover $Z = X \cup \{g,h\}$.
Observe that $Z\cap E_R = (F_1\cap F_2 \cap E_R)$.
Hence, $c(Z)\leq c(F_2)$ and
condition~\eqref{item:zCostOk} holds.
Furthermore, $X\subseteq F^1$, and thus
$b(Z) \leq b(F^1) + b(g) + b(h) \leq b(F^1) + 2b_{\max}$,
implying~\eqref{item:zBOk} and finishing the proof.
\end{proof}

Hence, $F^2$ corresponds to an interdiction strategy
that is optimal up to $2b_{\max}$.
From here, it is not hard to obtain a PTAS.
Let $\epsilon >0$. If $2b_{\max} \leq \epsilon b^T y^*$,
then $F^2$ corresponds to an interdiction strategy
that is an $(1-\epsilon)$-approximation.
Otherwise, we use the following well-known guessing
technique (see~\cite{ravi_1996_constrained,grandoni_2014_new}).
Consider an optimal integral solution
$\bar{y}, \bar{r}, \bar{r}_{IJ}$ of~\eqref{eq:budgetedEdgeCover}.
The vector $\bar{r}$ of such a solution is the characteristic
vector of an optimal interdiction set, and $\OPT = b^T \bar{y}$
is the optimal value of our interdiction problem.
We guess the $\lceil \frac{2}{\epsilon} \rceil$
heaviest edges $W$
of $\{e\in E \mid \bar{y}(e) = 1\}$,
i.e., the ones with highest $b$-values. This can be done
by going through all subsets of $E$ of size
$\lceil\frac{2}{\epsilon}\rceil$,
which is a polynomial number of subsets
for a fixed $\epsilon > 0$.
For each such guess we consider the resulting residual
version of problem~\eqref{eq:budgetedEdgeCover}, where
we set $y(e)=1$ for each guessed edge and remove all
edges of strictly higher $b$-values than the lowest
$b$-value of the guessed edges. 
Hence, we end up with a residual problem
where $b_{\max}$ is less than or equal to
the $b$-value of any guessed edge. For the right
guess $W$, we have $b(W)\leq \OPT$ and thus get 
\begin{equation*}
b_{\max} \leq \frac{\epsilon}{2}b(W)
  \leq \frac{\epsilon}{2} \OPT,
\end{equation*}
implying that the set $F^2$ for the right guess
is indeed a $(1-\epsilon)$-approximation.

Notice that if $b_{\max}$ is
sufficiently small with respect
to $b^T y^*$, i.e., $2b_{\max} \leq \epsilon b^T y^*$,
then the expensive guessing step can be skipped.

\subsection{Efficient algorithm for stable set interdiction
in bipartite graphs}
\label{sec:maxCardBIndepSet}

We complete the discussion on bipartite $b$-stable set
interdiction by showing that the problem of interdicting
stable sets, which are the same as $1$-stable sets,
in a bipartite graph can be solved in polynomial time.

We reuse the notation of the previous section. Hence,
$G=(V,E)$ is a bipartite graph with bipartition $V=I\cup J$,
$c:E\rightarrow \mathbb{Z}_{>0}$ are the interdiction
costs, and $B\in \mathbb{Z}_{>0}$ is the interdiction budget.
Furthermore, we denote by $\alpha(G)$ the size of a maximum
cardinality stable set in $G$ and by $\nu(G)$ the size
of a maximum cardinality matching.
It is well-known from K\"onig's Theorem that
for any bipartite graph $G=(V,E)$,
\begin{equation*}
\alpha(G) = |V| - \nu(G).
\end{equation*}
Hence, the objective value of some interdiction set
$R\subseteq V$ with $c(R) \leq B$ is equal to
\begin{equation*}
\alpha(G[V\setminus R]) = |V| - |R| - \nu(G[V\setminus R]),
\end{equation*}
where $G[W]$ for any $W\subseteq V$ is the induced subgraph
of $G$ over the vertices $W$, i.e., the graph obtained from
$G$ by removing $V\setminus W$.

We start by discussing some structural properties that can
be assumed to hold for at least one optimal solution.
Let $R^*$ be an optimal solution to the interdiction
problem, and let $M^*\subseteq E$ be a maximum
cardinality matching in $G[V\setminus R]$.
By the above discussion, the value of the interdiction
set $R^*$ is
\begin{equation}\label{eq:indepToMatching}
\alpha(G[V\setminus R^*]) = |V| - |R^*| - |M^*|.
\end{equation}

In the following, we will focus on finding an
optimal matching $M^*$, and then derive $R^*$
from this matching. We start with a lemma that
shows how $R^*$ can be obtained from $M^*$.
For this we need some additional notation.
We number the vertices 
$V=\{v_1,\dots, v_n\}$ such that
$c(v_1) \leq c(v_2) \leq \dots \leq c(v_n)$.
For $\ell\in \{0,\dots, n\}$ let
$V_{\ell}=\{v_1,\dots, v_{\ell}\}$ with
$V_0 = \emptyset$.
Furthermore, for any subset of edges $U$,
we denote by $V(U)$ the set~$\cup_{e\in U}e$ of all endpoints
of edges in $U$.

\begin{lemma}\label{lem:RFromM}
  Let $R$ be an optimal interdiction set and
 let $M^*$ be a maximum cardinality matching
in the graph $G[V\setminus R]$.
Then the set $R^*\subseteq V$ defined below
is also an optimal solution to the interdiction
problem.
\begin{equation*}
R^* = V_\ell \setminus V(M^*),
\end{equation*}
where $\ell\in \{0,\dots, n\}$ is the largest
value such that $c(V_\ell \setminus V(M^*))\leq B$.
\end{lemma}
\begin{proof}
   The interdiction set $R^*$ is budget feasible by assumption, and  $R,R^*\subseteq V\setminus V(M^*)$ so $|R|\leq |R^*|$ by the construction of $R^*$.
  Let $M'$ be a maximum cardinality matching in the graph $G[V\setminus R^*]$.
  Since $M^*$ is a matching in $G[V\setminus R^*]$ it holds that $|M'|\geq |M^*|$.
  Thus $R^*$ is also an optimal interdiction set because
  \[\alpha(G[V\setminus R^*]) = |V|-|R^*|-|M'| \leq |V|-|R|-|M^*|=\alpha(G[V\setminus R]).\]
\end{proof}

One of our key observations is that we can find
an optimal matching $M^*$ of Lemma~\ref{lem:RFromM}
efficiently by matroid intersection techniques
if we know the following four quantities
that depend on $M^*$:
\begin{enumerate}[(i)]
\item The maximum value $\ell\in \{0,\dots, n\}$
such that $R^*=V_\ell \setminus V(M^*)$ satisfies
$c(R^*)\leq B$;

\item\label{it:paramBetaI} $\beta_I = |V_\ell \cap V(M^*) \cap I|$;

\item\label{it:paramBetaJ} $\beta_J = |V_\ell \cap V(M^*) \cap J|$;

\item\label{it:paramGamma} $\gamma = |M^*|$.
\end{enumerate}

There may be different quadruples
$(\ell,\beta_I, \beta_J, \gamma)$
that correspond to different optimal
matchings $M^*$. However, we need any such set
of values that corresponds to an optimal $M^*$.
Before showing how an optimal quadruple
$(\ell,\beta_I, \beta_J, \gamma)$ can be used
to find $M^*$ by matroid intersection, we
highlight that there is only a polynomial
number of possible quadruples. This follows
since $\ell\in \{0,\dots, n\}$ can only take
$n+1$ different values, $\beta_I$ and $\beta_J$
only take at most $|I|+1$ and $|J|+1$ different
values, respectively, and the cardinality of
$M^*$ is between $0$ and
$\nu(G)\leq \min\{|I|,|J|\}$.
Hence, each possible quadruple
$(\ell, \beta_I, \beta_J, \gamma)$ is element
of the set
\begin{equation*}
\mathcal{Q} = \{0,\dots, n\} \times \{0,\dots, |I|+1\}
\times \{0,\dots, |J|+1\}
\times \{0,\dots, \nu(G)\}.
\end{equation*}
We will go through all quadruples in $\mathcal{Q}$
and try to construct a corresponding mathcing
$M^*$ by the matroid intersection technique
that we introduce below.
Thus, we will consider at least once an optimal
quadruple, for which we will obtain an optimal
$M^*$, which will then lead to an optimal $R^*$
through Lemma~\ref{lem:RFromM}.
Hence, our task reduces to find a matching
that ``corresponds'' to a given quadruple
in $\mathcal{Q}$. We define formally what
this means in the following.
\begin{definition}\label{def:corrMatching}
We say that a matching $M$ in $G$ \emph{corresponds}
to $(\ell,\beta_I, \beta_J, \gamma)\in \mathcal{Q}$
if the following conditions are
fulfilled:
\begin{enumerate}[(i)]
\item $c(V_{\ell}\setminus V(M))\leq B$,
\item $\beta_I = |V_\ell \cap V(M) \cap I|$,
\item $\beta_J = |V_\ell \cap V(M) \cap J|$,
\item $\gamma = |M|$.
\end{enumerate}
We call a quadruple in $\mathcal{Q}$ \emph{feasible}
if there exists a matching that corresponds
to it.
Furthermore, a quadruple is called \emph{optimal}
if there is a matching $M^*$ corresponding to
it such that $R^*=V_\ell \setminus V(M^*)$ is
an optimal interdiction set.
\end{definition}
Notice that our definition of a matching $M$
corresponding to a quadruple
$(\ell,\beta_I, \beta_J,\gamma)\in \mathcal{Q}$ does
not require that $\ell$ is the maximum value such
that $c(V_{\ell} \setminus M)\leq B$ since we obtain
the properties we need without requiring this condition
in our correspondence, as shown by the next lemma.

\begin{lemma}\label{lem:valFromParams}
Let $(\ell, \beta_I, \beta_J, \gamma)\in \mathcal{Q}$ be
a feasible quadruple with $M$ corresponding to it.
Then the set $R=V_\ell \setminus V(M)$ is an interdiction
set of objective value
\begin{equation}\label{eq:valFromParams}
\alpha(G[V\setminus R]) \leq 
|V| - |R| - |M| = 
|V| - \gamma - \ell + \beta_I + \beta_J.
\end{equation}
Furthermore, if
$(\ell, \beta_I, \beta_J, \gamma)\in \mathcal{Q}$ is
an optimal quadruple, then 
\begin{equation*}
\alpha(G[V\setminus R]) = |V| - \gamma - \ell
  + \beta_I + \beta_J.
\end{equation*}
\end{lemma}
\begin{proof}
The inequality in~\eqref{eq:valFromParams} follows immediately
from~\eqref{eq:indepToMatching} since
\begin{align*}
\alpha(G[V\setminus R]) &= |V| - |R| - \nu(G[V\setminus R]) &&
  \text{(by~\eqref{eq:indepToMatching})}\\
  &\leq |V| - |R| - |M| &&
   \text{(since $M$ is a matching in $G[V\setminus R]$)},
\end{align*}
with equality if and only if $M$ is a maximum cardinality matching in $G[V\setminus R]$.
To obtain the equality in~\eqref{eq:valFromParams}, we
observe that $|M|=\gamma$. Furthermore,
\begin{align*}
|R| = |V_\ell \setminus V(M)|
    = |V_\ell| - |V_\ell\cap V(M)\cap I|
               - |V_\ell \cap V(M) \cap J|
    = \ell - \beta_I - \beta_J,
\end{align*}
which implies the desired equality.
\end{proof}

The main consequence of Lemma~\ref{lem:valFromParams} is that the value of an optimal solution is determined entirely by its quadruple.
Thus one can find an optimal quadruple by testing the feasibility of every quadruple in $\mathcal{Q}$ and choosing one that minimizes the right hand side of~\eqref{eq:valFromParams}.
We conclude the following.
\begin{corollary}
Let $(\ell, \beta_I, \beta_J, \gamma)\in \mathcal{Q}$ be
a feasible quadruple that minimizes
$|V|-\gamma-\ell+\beta_I+\beta_J$, and let $M^*$ be
a matching that corresponds to it. Then
$R^*= V_\ell \setminus V(M^*)$ is an optimal interdiction
set to bipartite stable set interdiction problem.
\end{corollary}

Hence, all that remains to be done to obtain an
efficient algorithm is to design a procedure that,
for a quadruple
$(\ell, \beta_I, \beta_J, \gamma)\in \mathcal{Q}$
decides whether it is feasible, and if so, finds
a corresponding matching $M$.
Using this procedure we check all quadruples to
determine a feasible quadruple $(\ell,\beta_I,\beta_J,\gamma)$
that minimizes $|V|-\gamma-\ell-\beta_I-\beta_J$,
and then return $R^* = V_\ell \setminus V(M^*)$,
where $M^*$ is a matching corresponding to such
a quadruple.

Hence, let $q=(\ell,\beta_I,\beta_J,\gamma)\in Q$ and
we show how to check feasibility of $q$ and find
a corresponding matching $M$ if $q$ is feasible.
Let $c': E\rightarrow \mathbb{Z}_{\geq 0}$ be an
auxiliary cost function defined by
\begin{equation*}
c'(v_k) = \begin{cases}
  c(v_k)  &\text{if } k\leq \ell,\\
       0  &\text{if } k> \ell.
\end{cases}
\end{equation*}
Based on $c'$ we define weights
$w:E\rightarrow \mathbb{Z}_{\geq 0}$,
where
\begin{equation*}
w(\{i,j\}) = c'(i) + c'(j) \qquad \forall \{i,j\}\in E.
\end{equation*}
Our goal is to determine a maximum weight matching
$M$ in $G$ such that $|V_\ell \cap V(M) \cap I|=\beta_I$,
$|V_\ell \cap V(M) \cap J|=\beta_J$, and $\gamma=|M|$.
Notice that maximizing $w$ corresponds to
maximizing $c(V(M)\cap V_\ell)$.
Hence, a maximizer $M$ will be a matching
in $G$ that satisfies conditions \eqref{it:paramBetaI}-\eqref{it:paramGamma}
of Definition~\ref{def:corrMatching} and
subject to fulfilling these three
conditions, it maximizes $c(V(M)\cap V_\ell)$,
which is the same as minimizing $c(V_\ell\setminus V(M))$.
Hence, if $c(V_\ell\setminus V(M))\leq B$, then
the quadruple $(\ell, \beta_I, \beta_J, \gamma)$
is feasible and $M$ corresponds to it, otherwise,
the quadruple is not feasible. It remains to
show how to find efficiently a maximum weight
matching $M$ in $G$ such that
$|V_\ell \cap V(M) \cap I|=\beta_I$,
$|V_\ell \cap V(M) \cap J|=\beta_J$,
and $\gamma=|M|$.

This optimization problem
corresponds to maximizing $w$ over a face
of a matroid intersection polytope. Indeed,
define one laminar matroid $M_1=(E,\mathcal{F}_1)$ such
that a set $U\subseteq E$ is independent in $M_1$,
i.e., $U\in \mathcal{F}_1$, if $U$ contains at most
one edge incident with $i\in I$ for each $v\in I$,
at most $\beta_I$ edges incident with vertices
in $V_\ell \cap I$ and at most $\gamma$ edges in total.
Similarly, define $M_2=(E,\mathcal{F}_2)$ such
that $U\in \mathcal{I}_2$ if $U$ contains am most
one edge incident with any vertex $j\in J$, at
most $\beta_J$ edges incident with $V_\ell\cap J$
and at most $\gamma$ edges in total.
The problem we want to solve is to find
a set $M\in \mathcal{F}_1\cap \mathcal{F}_2$ such
that the constraints $|M\cap V_\ell \cap I|\leq \beta_I$,
$|M\cap V_\ell \cap J|\leq \beta_J$ and $|M|\leq \gamma$
are fulfilled with equality. Hence, this is indeed the
problem of maximizing $w$ over a particular face of
the matroid intersection polytope corresponding to
$M_1$ and $M_2$. This problem can be solved in strongly
polynomial time by matroid intersection algorithms.
Alternatively, one can also find a vertex solution
to the following polynolmial-sized LP,
which describes this face of
the matroid intersection polytope, and is therefore
integral when feasible.
\begin{equation*}
\begin{array}{rrcll}
\max & w^T x & & \\
     &  x(\delta(v)) & \leq & 1 & \forall v\in V\\
     &  x(\delta(V_\ell \cap I)) & = & \beta_I &\\
     &  x(\delta(V_\ell \cap J)) & = & \beta_J &\\
     &  x(E) & = & \gamma &
\end{array}
\end{equation*}
For more details on optimization
over the matroid intersection
polytope, we refer the interested reader
to~\cite[Chapter~41]{schrijver2003combinatorial}.

\section{Conclusions}
\label{sec:conclusions}

We present a framework to obtain $2$-pseudoapproximations
for a wide set of combinatorial interdiction problems,
including maximum cardinality independent set in a matroid
or the intersection of two matroids, maximum $s$-$t$ flows,
and packing problems defined by a constraint matrix that
is TU.
Our approach is inspired by a technique of
Burch et al.~\cite{burch2003decomposition}, who presented
a $2$-pseudoapproximation for maximum $s$-$t$ flows.
Furthermore, we show that our framework can also be adapted
to more general settings involving matroid optimization.
More precisely, we also get a $2$-pseudpapproximation
for interdicting the maximum weight independent set problem
in a matroid with submodular interdiction costs.
Submodularity is a natural property for interdiction
costs since it models economies of scale.
Our framework for $2$-pseudoapproximations is polyhedral
and sometimes we can exploit polyhedral properties of
well-structured interdiction problems to obtain stronger
results. We demonstrate this on the problem of
interdicting $b$-stable sets in bipartite graphs.
For this setting we obtain a PTAS, by employing ideas
from multi-budgeted optimization. Furthermore, we
show that the special case of stable set interdiction
in bipartite graphs can be solved efficiently by matroid
intersection techniques.

Many interesting open questions remain in the field
of interdicting combinatorial optimization problems.
It particular, it remains open whether stronger
pseudoapproximations can be obtained for the considered
problems.
Also in terms of ``true'' approximation algorithms,
large gaps remain.

\bibliographystyle{plain}
\bibliography{lit}

\end{document}